\documentclass[11pt,twoside]{amsart}

\usepackage{amssymb}
\usepackage{amsthm}
\usepackage{amsmath}
\usepackage{hyperref}
\usepackage{array}
\usepackage{ifthen}
\usepackage{mathtools}

\usepackage[utf8]{inputenc}   
\usepackage[T1]{fontenc}
\usepackage{textcomp}
\usepackage[english]{babel}

\usepackage{etoolbox}
\let\bbordermatrix\bordermatrix
\patchcmd{\bbordermatrix}{8.75}{4.75}{}{}
\patchcmd{\bbordermatrix}{\left(}{\left[}{}{}
\patchcmd{\bbordermatrix}{\right)}{\right]}{}{}

\usepackage{graphicx}
\usepackage{tikz}
\usetikzlibrary{arrows}
\usetikzlibrary{tikzmark}
\usetikzlibrary{shapes}
\usetikzlibrary{calc}

\theoremstyle{plain}
\newtheorem{theorem}{Theorem}

\newtheorem{corollary}[theorem]{Corollary}
\newtheorem{proposition}[theorem]{Proposition}

\theoremstyle{definition}

\theoremstyle{remark}

\newcommand{\set}[1]{\left\{#1\right\}}
\newcommand{\seqnum}[1]{\href{http://oeis.org/#1}{\underline{#1}}}

\def\tn{\textnormal}

\def\F{\mathcal{F}}

\def\M{\mathcal{M}}

\def\P{\mathcal{P}}
\def\Q{\mathcal{Q}}
\def\S{\mathcal{S}}
\def\T{\mathcal{T}}

\def\ce{\coloneqq}

\newcommand{\ignore}[1]{}

\DeclareMathOperator{\topt}{topt}
\DeclareMathOperator{\lopt}{lopt}
\DeclareMathOperator{\lofi}{lofi}
\DeclareMathOperator{\hdd}{hdd}
\DeclareMathOperator{\north}{north}
\DeclareMathOperator{\height}{ht}

\makeatletter
\renewcommand{\tocsection}[3]{%
\indentlabel{\@ifnotempty{#2}{\makebox[1.50em][l]{\ignorespaces#1#2.}}}#3}
\renewcommand{\tocsubsection}[3]{%
\indentlabel{\@ifnotempty{#2}{\hspace*{1.50em}\makebox[2.25em][l]{\ignorespaces#1#2.}}}#3}
\renewcommand{\tocsubsubsection}[3]{%
\indentlabel{\@ifnotempty{#2}{\hspace*{3.75em}\makebox[3.00em][l]{\ignorespaces#1#2.}}}#3}
\makeatother

\begin{document}

\title
[A new generalization of the Narayana numbers]
{A new generalization of the Narayana numbers inspired by linear operators 
on associative $d$-ary algebras}

\author{Yu Hin (Gary) Au}

\address{Department of Mathematics and Statistics, University of Saskatchewan, Saskatoon, Saskatchewan, S7N 5E6 Canada}

\email{au@math.usask.ca}

\author{Murray R. Bremner}

\address{Department of Mathematics and Statistics, University of Saskatchewan, Saskatoon, Saskatchewan, S7N 5E6 Canada}

\email{bremner@math.usask.ca}


\date{\today}


\begin{abstract}
We introduce and study a generalization of the Narayana numbers $N_d(n,k) = \frac{1}{n+1} \binom{n+1}{k+1} \binom{ n + (n-k)(d-2)+1}{k}$ for integers $d \geq 2$ and $n,k \geq 0$. This two-parameter array extends the classical Narayana numbers ($d=2$) and yields a $d$-ary analogue of the Catalan numbers $C_d(n) = \sum_{k=0}^n N_d(n,k)$. We give nine combinatorial interpretations of $N_d(n,k)$ that unify and generalize known combinatorial interpretations of the Narayana numbers and $C_3(n)$ in the literature. In particular, we show that $N_d(n,k)$ counts a natural class of operator monomials over a $d$-ary associative algebra, thereby extending a result of \cite{BremnerE22} for the binary case. We also construct explicit bijections between these monomials and several families of classic combinatorial objects, including Schr\"{o}der paths, Dyck paths, rooted ordered trees, and $231$-avoiding permutations.
\end{abstract}

\maketitle

{\footnotesize\tableofcontents}

\section{Introduction}

Given an integer $d \geq 2$ and integers $n,k \geq 0$, define
\[
N_d(n,k) \ce \frac{1}{n+1} \binom{n+1}{k+1} \binom{ n + (n-k)(d-2)+1}{k}.
\]
Table~\ref{tab101} lists $N_2(n,k)$ and $N_3(n,k)$ for several small values of $n$ and $k$.

\begin{table}[htbp]
$\begin{array}{l|rrrrrrrr}
n~\backslash~k & 0 & 1 & 2 & 3 & 4 & 5 & 6 & 7\\
\hline
0 &1 & & & & & & & \\
1 &1 &1 & & & & & & \\
2 &1 &3 &1 & & && &  \\
3 &1 &6 &6 &1 & & & & \\
4 & 1& 10& 20& 10&1 && &  \\
5 &1 &15 & 50 &50  &15 &1& &  \\
6 &1 & 21& 105& 175&105 &21 &1 & \\
7 &1 &28 &196 &490 &490 &196 &28 &1 \\
\end{array}
\quad
\begin{array}{l|rrrrrrrr}
n~\backslash~k & 0 & 1 & 2 & 3 & 4 & 5 & 6 & 7 \\
\hline
0 &1 & & & & & & & \\
1 &1 &1 & & & & & & \\
2 &1 &4 &1 & & && &  \\
3 &1 &9 &10 &1 & & & & \\
4 & 1& 16& 42& 20&1 && &  \\
5 &1 &25 & 120 &140  &35 &1& &  \\
6 &1 & 36& 275& 600 & 378 & 56 &1 & \\
7 &1 & 49 & 546 & 1925 & 2310 & 882 & 84 &1 
\end{array}
$
\medskip
\caption{The numbers $N_2(n,k)$ (left) and $N_3(n,k)$ (right) for $0 \leq k \leq n \leq 7$}\label{tab101}
\end{table}

When \( d = 2 \), we obtain the well-studied Narayana numbers (sequence \seqnum{A001263} from the On-line Encyclopedia of Integer Sequences (OEIS)~\cite{OEIS}), named after the 20th-century Indo-Canadian mathematician Tadepalli Venkata Narayana. For an introduction to the Narayana numbers, see, for instance,~\cite[Chapter 2]{Petersen15} or~\cite[Chapter 33]{Grimaldi12}. The sequences \( N_d(n,k) \) for \( d \geq 3\) were not in the OEIS at the start of this research project; however, \( N_3(n,k) \) was briefly mentioned in Callan and Mansour~\cite[Section 3.16]{CallanM18} and studied in more depth in Huh et al.~\cite{HuhKSS24}. In preparation for this manuscript, we have submitted the sequences $N_3(n,k)$ (\seqnum{A391045}), $N_4(n,k)$ (\seqnum{A391046}), $N_5(n,k)$ (\seqnum{A391047}), and $N_6(n,k)$ (\seqnum{A391048}) to the OEIS. 
(The sequence $N_3(n,k)$ has been applied by the second author \cite{Bremner25} 
in work on algebraic identities for linear operators.)
For several other generalizations of the Narayana numbers, see~\cite{Barry11, Callan22, KruchininKS22, Sulanke04}. Additionally, for the integer sequence named after the 14th-century Indian mathematician Narayana Pandita, which is sometimes also referred to as the Narayana numbers in the literature, see~\cite{AlloucheJ96} and \seqnum{A000930} from the OEIS.

It is well-known that each row of Narayana numbers sums to the familiar Catalan numbers (\seqnum{A000108}). Consequently, our generalized Narayana numbers naturally lead to a \( d \)-ary generalization of the Catalan numbers. More precisely, for \( d \geq 2 \), we define 
\[
C_d(n) \ce \sum_{k=0}^{n} N_d(n,k).
\]
Table~\ref{tab102} lists the first few entries of \( C_d(n) \) for \( d \leq 6 \). Notably, \( C_2(n) \) corresponds to the Catalan numbers, which admits over 200 combinatorial interpretations~\cite{Stanley15}. The sequence \( C_3(n) \) has also been studied in various combinatorial contexts~\cite{AlbertHPSV18, AsinowskiB24, Barry16, BeatonBGR19, BenyiMR24, BrycFM19, BursteinS22, CallanM18, CaoJL19, GaoK19, GuLM08, HuhKSS24, MartinezS18, MerinoM23, YanL20} --- see also \seqnum{A106228} for additional properties of this sequence. The sequences for $4 \le d \le 6$ appear in the OEIS but lack substantial combinatorial interpretation. As a consequence of our work on $N_d(n,k)$, this paper connects those sequences to a variety of combinatorial objects.

\begin{table}[htbp]
\[
\begin{array}{l|rrrrrrrrr}
d~\backslash~n &\quad 0 &\quad 1 &\quad 2 &\quad 3 &\quad 4 &\quad 5 &\quad 6 &\quad 7 &\quad \text{OEIS} \\
\cline{1-10}
\\[-10pt]
2  &\quad 
1 &\quad 2 &\quad 5 &\quad 14 &\quad 42 &\quad 132 &\quad 429 &\quad 1430 
&\quad \seqnum{A000108} \\
3  &\quad 
1 &\quad 2 &\quad 6 &\quad 21 &\quad 80 &\quad 322 &\quad 1347 &\quad 5798
&\quad \seqnum{A106228} \\
4 &\quad 
1 &\quad 2 &\quad 7 &\quad 29 &\quad 131 &\quad 627 &\quad 3124 &\quad 16032
&\quad  \seqnum{A300048} \\
5  &\quad 
1 &\quad 2 &\quad 8 &\quad 38 &\quad 196 &\quad 1073 &\quad 6120 &\quad 35968
&\quad \seqnum{A364723} \\
6  &\quad 
1 &\quad 2 &\quad 9 &\quad 48 &\quad 276 &\quad 1687 &\quad 10750 &\quad 70597
&\quad\seqnum{A364734}
\end{array}
\]
\caption{The generalized Catalan numbers $C_d(n)$}\label{tab102}
\end{table}

In this manuscript, we study the generalized Narayana numbers \( N_d(n,k) \) and present nine combinatorial interpretations of these numbers. First, in Section~\ref{sec2}, we demonstrate that \( N_d(n,k) \) counts a specific set of operator monomials over a \( d \)-ary algebra, generalizing a result by the second author and Elgendy~\cite{BremnerE22}, who established this for the case when \( d=2 \). 

In Section~\ref{sec3}, we provide bijections from these \( d \)-ary operator monomials to several classic combinatorial objects, showing that \( N_d(n,k) \) also counts certain subsets of Schr\"{o}der paths, rooted ordered trees, Dyck paths, and \( 231 \)-avoiding permutations.

In Section~\ref{sec4}, we provide four additional combinatorial interpretations of \( N_d(n,k) \) that generalize known interpretations for \( N_3(n,k) \) studied in Huh et al.~\cite{HuhKSS24}. These interpretations involve Schr\"{o}der paths with labelled descents, a family of lattice paths generalizing the \( F \)-paths defined in Huh et al., labelled Dyck paths, and labelled ordered trees. Our new interpretations also specialize to classic interpretations for Narayana and Catalan numbers when \( d=2 \). We conclude with some potential future research directions in Section~\ref{sec5}.

By studying \( N_d(n,k) \), we aim to generalize some classic results on Narayana and Catalan numbers. Moreover, under this framework, the well-studied sequence \seqnum{A106228} can now be seen as a ternary extension of the classic Catalan numbers, offering new perspectives on this and related sequences. The mapping we provide between \( d \)-ary operator monomials and other well-studied combinatorial objects may also be beneficial in exploring related problems in algebraic operads.
In particular, since the $d$-ary operator monomials form a basis for the nonsymmetric operad 
generated by one unary operation and one associative $d$-ary operation, 
the bijections we construct in this paper 
between these operator monomials and other combinatorial objects will aid in the understanding of
the $d$-ary analogues of well-known binary operads.

\section{$d$-ary operator monomials}\label{sec2}

In this section, we provide the first combinatorial interpretation of $N_d(n,k)$;
namely, the interpretation in terms of $d$-ary operator monomials. 

Let $\mathcal{A}$ be an associative $d$-ary algebra.
Specifically, $\mathcal{A}$ is a vector space equipped with a multilinear map 
$B \colon \mathcal{A}^d \to \mathcal{A}$
with product denoted $B( a_1, \dots, a_d ) \longmapsto a_1 \cdots a_d$ 
that satisfies $d$-ary associativity. That is, for every $a_1, \ldots, a_{2d-1} \in \mathcal{A}$ and indices $0 \le i < j \le d-1$, 
\[
a_1 \cdots a_i ( a_{i+1} \cdots a_{i+d} ) a_{i+d+1} \cdots a_{2d-1}
=
a_1 \cdots a_j ( a_{j+1} \cdots a_{j+d} ) a_{j+d+1} \cdots a_{2d-1}.
\]
It is easy to see that any monomial in a $d$-ary operation must have $k(d-1)+1$ indeterminates
for some $k \ge 0$; that is, the number of indeterminates must be congruent to 1 modulo $d-1$.
(Thus for $d=2$ this is no restriction, and for $d=3$ it says that the number of indeterminates
must be odd.)
The $d$-ary associativity identities imply by the usual inductive argument on $k$
that in any $d$-ary monomial we may unambiguously remove the parentheses.

Associative $d$-ary algebras were first studied by Carlsson \cite{Carlsson};
for a more modern operadic point of view see Gnedbaye \cite{Gnedbaye}.
For background on operads, see Loday and Vallette \cite{LodayVallette} for the theoretical aspects,
and Bremner and Dotsenko \cite{BremnerDotsenko} for the algorithmic aspects.

Next, let $L \colon \mathcal{A} \to \mathcal{A}$ be a linear operator (unary operation) on 
the underlying vector space of the associative $d$-ary algebra $\mathcal{A}$. 
We define the set of \emph{($d$-ary) operator monomials} \( \mathcal{M}_d \) recursively as follows:
\begin{itemize}
\item 
a single indeterminate (regarded as a generic element of $\mathcal{A}$) is an operator monomial;
\item 
if $M_1, \dots, M_d$ are operator monomials, then $M_1 \cdots M_d$ 
(using the $d$-ary associative product in $\mathcal{A}$) is also an operator monomial;
\item if $M$ is an operator monomial, then $L(M)$ is also an operator monomial.
\end{itemize}

Given an operator monomial \( M \in \mathcal{M}_d \), we let \( \topt(M) \) (for \emph{total operations}) denote the total number of operations which appear (counting both \( B \) and \( L \)), and \( \lopt(M) \) (for \emph{linear operations}) denote the number of occurrences of the linear operator \( L \). 
For example, Table~\ref{tab201} lists the ternary (\( d=3 \)) operator monomials \( M \) with \( \topt(M) = 3 \), where we use the standard abbreviations such as $L(L(a)) = L^2(a)$.

\begin{table}[htbp]
\centering
\[
\begin{array}{lcc}
\tn{operator monomials $M$} &\lopt(M) & \tn{count} 
\\[4pt]
\hline
\\[-10pt]
 a_1a_2a_3a_4a_5a_6a_7
& 0
& 1 
\\[4pt]
\hline
\\[-10pt]
L(a_1a_2a_3a_4a_5), \quad 
L(a_1a_2a_3)a_4a_5, \quad 
L(a_1)a_2a_3a_4a_5, \quad 
a_1L(a_2a_3a_4)a_5, \quad 
& 1 
& 9 
\\
\\[-10pt]
a_1L(a_2)a_3a_4a_5, \quad
a_1a_2L(a_3a_4a_5), \quad
a_1a_2L(a_3)a_4a_5, \quad
a_1a_2a_3L(a_4)a_5, \quad 
&& 
\\
\\[-10pt]
a_1a_2a_3a_4L(a_5) 
&&
\\[4pt]
\hline
\\[-10pt]
L^2(a_1a_2a_3), \quad 
L(L(a_1)a_2a_3), \quad
L^2(a_1)a_2a_3, \quad 
L(a_1L(a_2)a_3), \quad
& 2
& 10 
\\ 
\\[-10pt]
L(a_1a_2L(a_3)), \quad
L(a_1)L(a_2)a_3, \quad 
L(a_1)a_2L(a_3), \quad 
a_1L^2(a_2)a_3, \quad 
&& 
\\
\\[-10pt]
a_1L(a_2)L(a_3), \quad   
a_1a_2L^2(a_3)
&&
\\[4pt]
\hline
\\[-10pt]
L^3(a_1)
& 3 
& 1 
\end{array}
\]
\medskip
\caption{The 21 ternary operator monomials $M \in \M_3$ with $\topt(M) =3$}
\label{tab201}
\end{table}

Observe that for every \( M \in \mathcal{M}_d \), we have \( 0 \leq \lopt(M) \leq \topt(M) \). Next, we let \( \deg(M) \) denote the \emph{degree} of \( M \), which is the number of indeterminates involved in \( M \). It follows that
\begin{equation}\label{eqdegM}
\deg(M) = (\topt(M) - \lopt(M))(d-1) + 1
\end{equation}
for every \( M \in \mathcal{M}_d \). 
To understand this, note that \( M \) contains exactly \( \topt(M) - \lopt(M) \) instances of the 
associative \( d \)-ary operation, and each application of that operation ``multiplies down'' \( d \) 
inputs into one output. An immediate consequence of~\eqref{eqdegM} is that $\deg(M)$ is congruent to \( 1 \) modulo $d-1$ for every \( M \in \mathcal{M}_d \). In fact, to produce all operator monomials of degree $k(d-1)+1$ (for any given integer $k \geq 1$), we can start with the generic monomial $a_1 \cdots a_{k(d-1)+1}$ and insert  operator symbols $L$ in all possible ways, respecting the fact that any product must contain a number of factors congruent to 1 modulo $d-1$.

Next, we show that the set of monomials in \( \mathcal{M}_d \) with a fixed number of operations is counted by \( N_d(n,k) \).

\begin{theorem}\label{thm201}
Let $d \geq 2$ and $n,k \geq 0$ be integers. Then \( N_d(n,k) \) counts the number of \( d \)-ary operator monomials \( M \in \mathcal{M}_d \) such that \( \topt(M) = n \) and \( \lopt(M) = k \).
\end{theorem}

Theorem~\ref{thm201} extends a result by the second author and Elgendy~\cite[Lemma 2.5]{BremnerE22}, who proved the result for the case \(d=2\). Before we prove Theorem~\ref{thm201}, we first derive a functional equation for the generating function of \( N_d(n,k) \). Let
\[
A_d(x,y) \ce \sum_{n \geq 0} \sum_{k=0}^n N_d(n,k)x^n y^k.
\]
For example, for \( d=2 \), we have
\[
A_2(x,y) = 1 + (1+y)x + (1+3y+y^2)x^2 + (1+6y+6y^2+y^3)x^3 + \cdots
\]
We then have the following.

\begin{proposition}\label{prop202}
For every integer \( d \geq 2 \),
\begin{equation}\label{prop202eq0}
A_d(x,y) = \left( 1 + xy A_d(x,y) \right) \left( 1 + x A_d(x,y) \left( 1 + xy A_d(x,y) \right)^{d-2} \right).
\end{equation}
\end{proposition}

\begin{proof}
Let \( u \) be a function of \( x \) and \( y \) that satisfies the functional equation
\begin{equation}\label{prop202eq1}
u = (1 + xy u) \left( 1 + xu (1 + xy u)^{d-2} \right).
\end{equation}
Next, let \( v \ce xy u \). Multiplying both sides of~\eqref{prop202eq1} by \( xy \) and simplifying, we obtain
\begin{equation}\label{prop202eq2}
v = x(1 + v)(y + v(1 + v)^{d-2}).
\end{equation}
To prove our claim, it suffices to show that \( [x^{n+1}][y^{k+1}] v = N_d(n,k) \) for every \( n, k \geq 0 \) and \( d \geq 2 \). Observe that~\eqref{prop202eq2} can be rewritten as \( v = x \Phi(v) \), where \( \Phi(t) \ce (1 + t)\left( y + t(1 + t)^{d-2} \right) \). Now, \( \Phi(0) = y \), which we may assume to be nonzero. We will now proceed to obtain the coefficients of \( v \) via Lagrange inversion (see, for instance,~\cite[Chapter 5]{Wilf06} for a reference).

\begin{align*}
[x^{n+1}] [y^{k+1}] v
&= [y^{k+1}] [t^{n}] \frac{1}{n+1} \Phi(t)^{n+1} \\
&= [y^{k+1}] [t^{n}] \frac{1}{n+1} (1 + t)^{n+1} \left( y + t(1 + t)^{d-2} \right)^{n+1} \\
&= [y^{k+1}] [t^{n}] \frac{1}{n+1} (1 + t)^{n+1} \left( \sum_{i = 0}^{n+1} \binom{n+1}{i} y^i t^{n+1-i} (1 + t)^{(n+1-i)(d-2)} \right) \\
&= [y^{k+1}] [t^{n}] \frac{1}{n+1} \left( \sum_{i = 0}^{n+1} \binom{n+1}{i} y^i t^{n+1-i} (1 + t)^{n + (n+1-i)(d-2) + 1} \right) \\
&= [t^{n}] \frac{1}{n+1} \binom{n+1}{k+1} t^{n-k} (1 + t)^{n + (n-k)(d-2) + 1} \\
&= [t^{k}] \frac{1}{n+1} \binom{n+1}{k+1} (1 + t)^{n + (n-k)(d-2) + 1} \\
&= \frac{1}{n+1} \binom{n+1}{k+1} \binom{n + (n-k)(d-2) + 1}{k} \\
&= N_d(n,k).
\end{align*}
This concludes the proof.
\end{proof}

Next, given \( M \in \mathcal{M}_d \), we say that \( M \) is \emph{irreducible} if it satisfies one of the following conditions:
\begin{itemize}
    \item It is a single indeterminate (in which case \( \topt(M) = \lopt(M) = 0 \)), or
    \item \( M = L(M') \) for some operator monomial \( M' \in \mathcal{M}_d \).
\end{itemize}
We let \( \overline{\mathcal{M}}_d \subseteq \mathcal{M}_d \) denote the set of irreducible \( d \)-ary operator monomials. Additionally, given a positive integer \( n \), we let \( [n] \) denote the set \( \{1, 2, \ldots, n\} \). We are now ready to prove Theorem~\ref{thm201}.

\begin{proof}[Proof of Theorem~\ref{thm201}]
Let \( d \geq 2 \) be fixed, and consider the generating function 
\[
u \ce \sum_{M \in \mathcal{M}_d} x^{\topt(M)} y^{\lopt(M)}.
\]
To prove our claim, we will show that \( u \) satisfies the same functional equation as \( A_d(x,y) \) from Proposition~\ref{prop202}, which would imply that \( [x^{n}][y^{k}] u = N_d(n,k) \) for every integer \( n, k \geq 0 \). To derive a functional equation for \( u \), we first observe that
\begin{equation}\label{thm201eq1}
\sum_{M \in \overline{\mathcal{M}}_d} x^{\topt(M)} y^{\lopt(M)} = 1 + xyu.
\end{equation}
To see this, note that if \( M \in \overline{\mathcal{M}}_d \) is a single indeterminate, then \( M \) contributes 
\[
x^{\topt(M)} y^{\lopt(M)} = x^0 y^0 = 1
\]
to the sum. Otherwise, \( M \in \overline{\mathcal{M}}_d' \ce \{ L(M') : M' \in \mathcal{M}_d \} \), and we have
\[
\sum_{M \in \overline{\mathcal{M}}_d'} x^{\topt(M)} y^{\lopt(M)} = \sum_{M' \in \mathcal{M}_d} x^{\topt(M') + 1} y^{\lopt(M') + 1} = xyu.  
\]
Now, given an integer \( \ell \geq 1 \), let \( \mathcal{M}_{d, \ell} \subseteq \mathcal{M}_d \) be the set of operator monomials that can be expressed as the product of \( \ell \) irreducible operator monomials. Notice that every \( M \in \mathcal{M}_d \) is contained in \( \mathcal{M}_{d, j(d-1) + 1} \) for a unique integer \( j \geq 0 \). Thus, \( \mathcal{M}_d \) is equal to the disjoint union \( \bigcup_{j \geq 0} \mathcal{M}_{d, j(d-1) + 1} \). 

Now, given $M \in \M_d$, if we write $M =M_1M_2 \cdots M_{j(d-1)+1}$ where $M_i$ is irreducible for every $i \in [j(d-1)+1]$, then
\begin{equation}\label{thm201eq2}
\topt(M) = j + \sum_{i=1}^{j(d-1)+1} \topt(M_i), \quad \lopt(M) = \sum_{i=1}^{j(d-1)+1} \lopt(M_i).
\end{equation}
Combining~\eqref{thm201eq1} and~\eqref{thm201eq2}, we see that, for every $j \geq 0$, we have
\[
\sum_{M \in \M_{d, j(d-1)+1}} x^{\topt(M)} y^{\lopt(M)} = x^j \left(1+xyu \right)^{j(d-1)+1}.
\]
Putting everything together, we obtain that
\begin{align*}
u &= \sum_{M \in \M_d} x^{\topt(M)}y^{\lopt(M)} \\
&= \sum_{j \geq 0} \sum_{M \in \M_{d, j(d-1)+1}} x^{\topt(M)}y^{\lopt(M)} \\
&= \sum_{j \geq 0} x^j \left( 1+xyu \right)^{j(d-1)+1} \\
&= \frac{1+xyu}{1- x(1+xyu)^{d-1}}.
\end{align*}
Next, observe that
\begin{align*}
u &= \frac{1+xyu}{1-x(1+xyu)^{d-1}},\\
u - xu(1+xyu)^{d-1} &= 1+xyu,\\
u &= 1+xyu + xu(1+xyu)^{d-1}, \\
u &= (1+xyu) \left( 1+xu (1+xyu)^{d-2}\right).
\end{align*}
Observe that  $u$ indeed satisfies the same functional equation~\eqref{prop202eq0} as $A_d(x,y)$. Thus, we obtain that  $[x^{n}][y^{k}] u = N_d(n,k)$, and our claim follows.
\end{proof}

It is obvious that if we set $y=1$ in $A_d(x,y)$, we obtain the generating function for our generalized Catalan numbers $C_d(n)$. Thus, Proposition~\ref{prop202} and the proof of Theorem~\ref{thm201} readily imply the following.

\begin{corollary}\label{cor202b}
Let $d \geq 2$ be an integer, and define the generating function $u \ce \sum_{n \geq 0} C_d(n) x^n$. Then $u$ satisfies the functional equation
\[
u = (1+xu) \left( 1+xu (1+xu)^{d-2}\right),
\]
or equivalently,
\[
u = \frac{1+xu}{1-x(1+xu)^{d-1}}.
\]
\end{corollary}

We next highlight a simple identity involving \( N_d(n,k) \).

\begin{proposition}\label{prop203}
Let \( d \geq 2 \) and \( n, k \geq 0 \) be integers. Then 
\[
N_2(n,k) \leq N_d(n,k) \leq N_2((n-k)(d-1) + k, k).
\]
\end{proposition}

\begin{proof}
While it is relatively straightforward to derive both inequalities algebraically from the definition of \( N_d(n,k) \), we present a combinatorial argument. We first prove \( N_2(n,k) \leq N_d(n,k) \) by establishing an injection from \( \mathcal{M}_2 \) to \( \mathcal{M}_d \). Given \( M \in \mathcal{M}_2 \), let \( a_1, \ldots, a_{\ell} \) be the indeterminates involved in \( M \). Let \( M' \) be the operator monomial obtained from \( M \) by replacing every instance of \( a_i \) with a product of \( d \) indeterminates \( a_{i1} a_{i2} \cdots a_{id} \) for every \( i \in \{2, \ldots, \ell\} \). Then \( M' \in \mathcal{M}_d \) with \( \topt(M') = \topt(M) \) and \( \lopt(M') = \lopt(M) \), and the inequality follows.

Next, we prove \( N_d(n,k) \leq N_2((n-k)(d-1) + k, k) \) with an injection from \( \mathcal{M}_d \) to \( \mathcal{M}_2 \). Let \( M \in \mathcal{M}_d \) where \( \topt(M) = n \) and \( \lopt(M) = k \). Then \( \deg(M) = 1 + (n-k)(d-1) \). Now let \( M' \in \mathcal{M}_2 \) be the same expression as \( M \), but parsed using a binary operation instead of a \( d \)-ary operation. Then \( \lopt(M') = \lopt(M) = k \), and \( \topt(M') = (n-k)(d-1) + k \). This completes the proof.
\end{proof}

Proposition~\ref{prop203} offers two avenues for finding combinatorial interpretations of \( N_d(n,k) \). First, we can start with a set of combinatorial objects with cardinality \( N_2((n-k)(d-1) + k, k) \), and then \emph{restrict} it to a particular subset of size \( N_d(n,k) \). All four combinatorial interpretations we provide for \( N_d(n,k) \) in Section~\ref{sec3} belong to this category.

Alternatively, one could begin with a set of size \( N_2(n,k) \) and then \emph{replicate} its elements (e.g., by introducing labels to differentiate between them) to create an expanded set of size \( N_d(n,k) \). Three of the four combinatorial interpretations discussed in Section~\ref{sec4} (see Theorems~\ref{thm402}, \ref{thm403}, and \ref{thm404} in particular) can be viewed from this perspective.

Since the next two sections will involve several sets of combinatorial objects and mappings between them, we have organized how the bijections \( f_1, \ldots, f_8 \) relate to these combinatorial objects in Figure~\ref{fig201} for ease of reference. These objects and the bijections will be fully defined in Sections~\ref{sec3} and~\ref{sec4}.

\begin{center}
\begin{figure}[htbp]
\begin{tikzpicture}[xscale=2, yscale=1, >=latex', main node/.style={font=\normalsize}, word node/.style={font=\scriptsize}]


\def\x{0}
\node[main node] at (0,0) (M) {$\M_d$};
\node[main node] at (-1,1) (S) {$\S_d$};
\node[main node] at (-1,-1) (S') {$\widetilde{\S}_d$};
\node[main node] at (1,1) (T) {$\T_d$};
\node[main node] at (2,1) (Q) {$\Q_d$};
\node[main node] at (3,1) (P) {$\P_d$};
\node[main node] at (1,-1) (F) {$\F_d$};
\node[main node] at (2,-1) (Q') {$\widetilde{\Q}_d$};
\node[main node] at (3,-1) (T') {$\widetilde{\T}_d$};

\draw[->] (M) -- node[midway, above] {$f_1$} (S);
\draw[->] (M) -- node[midway, above] {$f_5$} (S');
\draw[->] (M) -- node[midway, above] {$f_2$} (T);
\draw[->] (T) -- node[midway, above] {$f_3$} (Q);
\draw[->] (Q) -- node[midway, above] {$f_4$} (P);
\draw[->] (M) -- node[midway, above] {$f_6$} (F);
\draw[->] (F) -- node[midway, above] {$f_7$} (Q');
\draw[->] (Q') -- node[midway, above] {$f_8$} (T');

\end{tikzpicture}
\caption{Relating combinatorial objects and associated bijections}\label{fig201}
\end{figure}
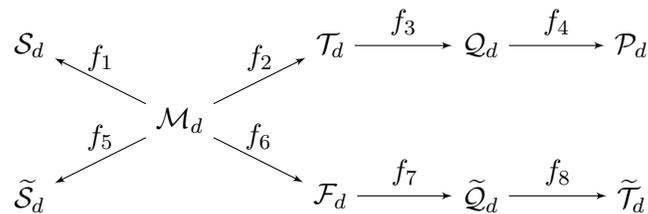
\end{center}

\section{Several ``restricted'' combinatorial interpretations of $N_d(n,k)$}\label{sec3}

In this section, we provide four sets of combinatorial objects that are counted by \( N_d(n,k) \) by establishing bijections among these sets and the \( d \)-ary operator monomials \( \mathcal{M}_d \).

\subsection{Schr\"{o}der paths with restricted horizontal steps}

A \emph{Schr\"{o}der path} is a lattice path that begins at \( (0,0) \), consists only of up steps \( U = (1,1) \), down steps \( D = (1,-1) \), and horizontal steps \( H = (2,0) \), remains on or above the \( x \)-axis, and ends on the \( x \)-axis. Observe that a Schr\"{o}der path \( P \) must end at \( (2n,0) \) for some non-negative integer \( n \), which is often referred to as the \emph{semilength} of \(P\) and denoted \( |P| \). Additionally, given an up step \( U \) in a Schr\"{o}der path that goes from \( y=\ell \) to \( y=\ell+1 \), we say that the \emph{matching} down step of \(U\) is the first instance of a down step \( D \) occurring after \(U\)  which goes from \( y=\ell+1 \) to \( y=\ell \). For example, for the Schr\"{o}der path
\[
P \ce U_1U_2U_3HD_1U_4HHD_2HD_3D_4HU_5D_5
\]
(where each up and down step is labelled by their order of occurrence), the matching down steps of \(U_1\), \(U_2\), \(U_3\), \(U_4\), and \(U_5\) are \(D_4\), \(D_3\), \(D_1\), \(D_2\), and \(D_5\) respectively.

Next, let \( \mathcal{S}_d \) be the set of nonempty Schr\"{o}der paths such that:
\begin{itemize}
    \item the total number of \( H \) steps in the path is equal to \( 1 + j(d-1) \) for some integer \( j \geq 0 \);
    \item between every up step \( U \) and its matching down step \( D \), the number of \( H \) steps between \( U \) and \( D \) is equal to \( 1 + j(d-1) \) for some integer \( j \geq 0 \).
\end{itemize}

We describe a simple bijection \( f_1 : \mathcal{M}_d \to \mathcal{S}_d \) defined as follows. Given \( M \in \mathcal{M}_d \), we scan across the expression \( M \) from left to right and:
\begin{itemize}
    \item replace every instance of \( L( \) with an up step \( U \);
    \item replace every indeterminate \( a_i \) with a horizontal step \( H \);
    \item replace every instance of \( ) \) with a down step \( D \).
\end{itemize}

For example, given \( M = L(a_1 L^2(a_2) a_3) a_4 a_5 \in \mathcal{M}_3 \), we have
\[
f_1(M) = UHUUHDDHDHH,
\]
which is indeed in \( \mathcal{S}_3 \). Figure~\ref{fig301} illustrates the mapping \( f_1 \) for all 21 monomials \( M \in \mathcal{M}_3 \) with \( \topt(M) = 3 \).

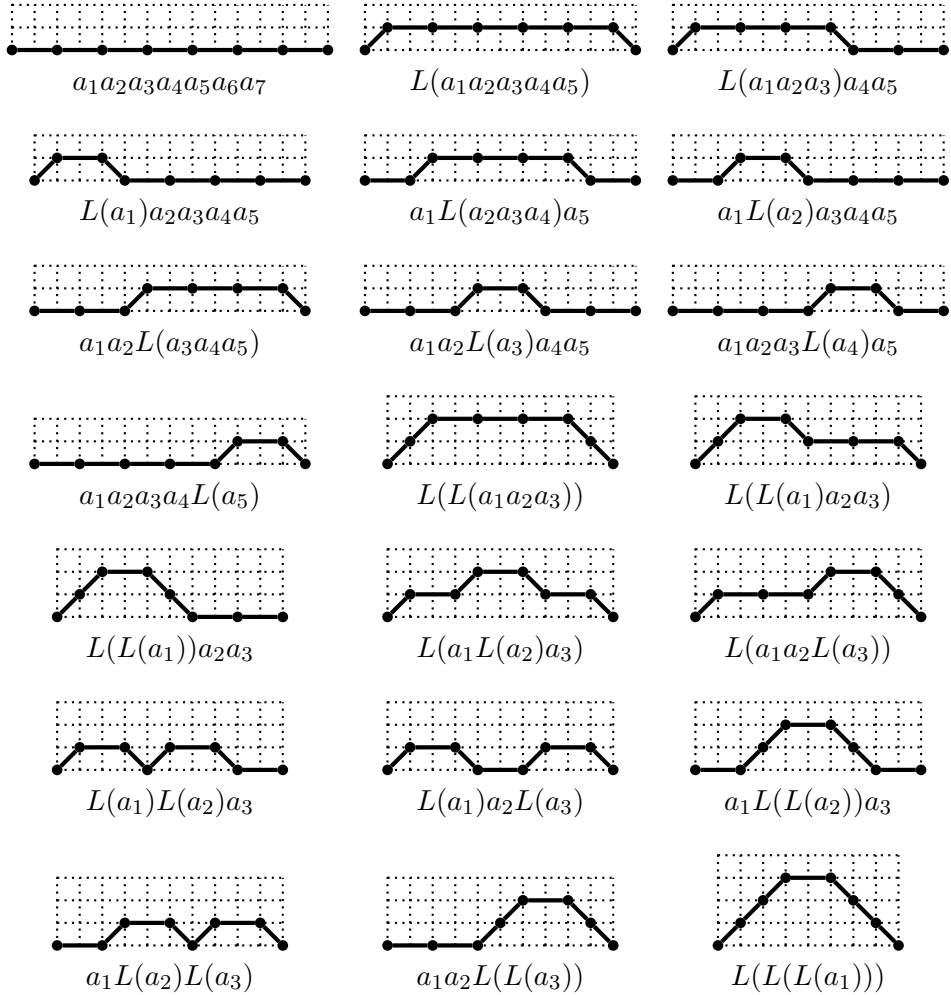
\begin{figure}[htbp]
\centering
\def\sc{0.3}
\def\grid{\foreach \x in {\xlb ,...,\xub}
{    \ifthenelse{\NOT 0 = \x}{\draw[thick](\x ,-1pt) -- (\x ,1pt);}{}
\draw[dotted, thick](\x,\ylb- \buf) -- (\x,\yub + \buf);}
\foreach \y in {\ylb ,...,\yub}
{    \ifthenelse{\NOT 0 = \y}{\draw[thick](-1pt, \y) -- (1pt, \y);}{}
\draw[dotted, thick](\xlb- \buf, \y) -- (\xub + \buf, \y);}
}

\begin{tabular}{ccc}
\begin{tikzpicture}[scale = \sc, font=\scriptsize\sffamily, thick,main node/.style={circle,inner sep=0.4mm,draw, fill}]
\def\xlb{0}; \def\xub{14}; \def\ylb{0}; \def\yub{2}; \def\buf{0}; \grid
\node[main node] at (0,0) (0) {};
\node[main node] at (2,0) (1) {};
\node[main node] at (4,0) (2) {};
\node[main node] at (6,0) (3) {};
\node[main node] at (8,0) (4) {};
\node[main node] at (10,0) (5) {};
\node[main node] at (12,0) (6) {};
\node[main node] at (14,0) (7) {};
\draw[ultra thick] (0) -- (1) -- (2) -- (3)-- (4)-- (5)-- (6)-- (7);
\end{tikzpicture}
&

\begin{tikzpicture}[scale = \sc, font=\scriptsize\sffamily, thick,main node/.style={circle,inner sep=0.4mm,draw, fill}]
\def\xlb{0}; \def\xub{12}; \def\ylb{0}; \def\yub{2}; \def\buf{0}; \grid
\node[main node] at (0,0) (0) {};
\node[main node] at (1,1) (1) {};
\node[main node] at (3,1) (2) {};
\node[main node] at (5,1) (3) {};
\node[main node] at (7,1) (4) {};
\node[main node] at (9,1) (5) {};
\node[main node] at (11,1) (6) {};
\node[main node] at (12,0) (7) {};
\draw[ultra thick] (0) -- (1) -- (2) -- (3)-- (4)-- (5)-- (6)-- (7);

\end{tikzpicture}
&
\begin{tikzpicture}[scale = \sc, font=\scriptsize\sffamily, thick,main node/.style={circle,inner sep=0.4mm,draw, fill}]
\def\xlb{0}; \def\xub{12}; \def\ylb{0}; \def\yub{2}; \def\buf{0}; \grid
\node[main node] at (0,0) (0) {};
\node[main node] at (1,1) (1) {};
\node[main node] at (3,1) (2) {};
\node[main node] at (5,1) (3) {};
\node[main node] at (7,1) (4) {};
\node[main node] at (8,0) (5) {};
\node[main node] at (10,0) (6) {};
\node[main node] at (12,0) (7) {};
\draw[ultra thick] (0) -- (1) -- (2) -- (3)-- (4)-- (5)-- (6)-- (7);
\end{tikzpicture}\\

$a_1a_2a_3a_4a_5a_6a_7$ &
$L(a_1a_2a_3a_4a_5)$ &
$L(a_1a_2a_3)a_4a_5$\\
\\

\begin{tikzpicture}[scale = \sc, font=\scriptsize\sffamily, thick,main node/.style={circle,inner sep=0.4mm,draw, fill}]
\def\xlb{0}; \def\xub{12}; \def\ylb{0}; \def\yub{2}; \def\buf{0}; \grid
\node[main node] at (0,0) (0) {};
\node[main node] at (1,1) (1) {};
\node[main node] at (3,1) (2) {};
\node[main node] at (4,0) (3) {};
\node[main node] at (6,0) (4) {};
\node[main node] at (8,0) (5) {};
\node[main node] at (10,0) (6) {};
\node[main node] at (12,0) (7) {};
\draw[ultra thick] (0) -- (1) -- (2) -- (3)-- (4)-- (5)-- (6)-- (7);
\end{tikzpicture}
&
\begin{tikzpicture}[scale = \sc, font=\scriptsize\sffamily, thick,main node/.style={circle,inner sep=0.4mm,draw, fill}]
\def\xlb{0}; \def\xub{12}; \def\ylb{0}; \def\yub{2}; \def\buf{0}; \grid
\node[main node] at (0,0) (0) {};
\node[main node] at (2,0) (1) {};
\node[main node] at (3,1) (2) {};
\node[main node] at (5,1) (3) {};
\node[main node] at (7,1) (4) {};
\node[main node] at (9,1) (5) {};
\node[main node] at (10,0) (6) {};
\node[main node] at (12,0) (7) {};
\draw[ultra thick] (0) -- (1) -- (2) -- (3)-- (4)-- (5)-- (6)-- (7);
\end{tikzpicture}
&
\begin{tikzpicture}[scale = \sc, font=\scriptsize\sffamily, thick,main node/.style={circle,inner sep=0.4mm,draw, fill}]
\def\xlb{0}; \def\xub{12}; \def\ylb{0}; \def\yub{2}; \def\buf{0}; \grid
\node[main node] at (0,0) (0) {};
\node[main node] at (2,0) (1) {};
\node[main node] at (3,1) (2) {};
\node[main node] at (5,1) (3) {};
\node[main node] at (6,0) (4) {};
\node[main node] at (8,0) (5) {};
\node[main node] at (10,0) (6) {};
\node[main node] at (12,0) (7) {};
\draw[ultra thick] (0) -- (1) -- (2) -- (3)-- (4)-- (5)-- (6)-- (7);
\end{tikzpicture}
\\

$L(a_1)a_2a_3a_4a_5$&
$a_1L(a_2a_3a_4)a_5$&
$ a_1L(a_2)a_3a_4a_5$\\
\\

\begin{tikzpicture}[scale = \sc, font=\scriptsize\sffamily, thick,main node/.style={circle,inner sep=0.4mm,draw, fill}]
\def\xlb{0}; \def\xub{12}; \def\ylb{0}; \def\yub{2}; \def\buf{0}; \grid
\node[main node] at (0,0) (0) {};
\node[main node] at (2,0) (1) {};
\node[main node] at (4,0) (2) {};
\node[main node] at (5,1) (3) {};
\node[main node] at (7,1) (4) {};
\node[main node] at (9,1) (5) {};
\node[main node] at (11,1) (6) {};
\node[main node] at (12,0) (7) {};
\draw[ultra thick] (0) -- (1) -- (2) -- (3)-- (4)-- (5)-- (6)-- (7);
\end{tikzpicture}
&
\begin{tikzpicture}[scale = \sc, font=\scriptsize\sffamily, thick,main node/.style={circle,inner sep=0.4mm,draw, fill}]
\def\xlb{0}; \def\xub{12}; \def\ylb{0}; \def\yub{2}; \def\buf{0}; \grid
\node[main node] at (0,0) (0) {};
\node[main node] at (2,0) (1) {};
\node[main node] at (4,0) (2) {};
\node[main node] at (5,1) (3) {};
\node[main node] at (7,1) (4) {};
\node[main node] at (8,0) (5) {};
\node[main node] at (10,0) (6) {};
\node[main node] at (12,0) (7) {};
\draw[ultra thick] (0) -- (1) -- (2) -- (3)-- (4)-- (5)-- (6)-- (7);
\end{tikzpicture}
&
\begin{tikzpicture}[scale = \sc, font=\scriptsize\sffamily, thick,main node/.style={circle,inner sep=0.4mm,draw, fill}]
\def\xlb{0}; \def\xub{12}; \def\ylb{0}; \def\yub{2}; \def\buf{0}; \grid
\node[main node] at (0,0) (0) {};
\node[main node] at (2,0) (1) {};
\node[main node] at (4,0) (2) {};
\node[main node] at (6,0) (3) {};
\node[main node] at (7,1) (4) {};
\node[main node] at (9,1) (5) {};
\node[main node] at (10,0) (6) {};
\node[main node] at (12,0) (7) {};
\draw[ultra thick] (0) -- (1) -- (2) -- (3)-- (4)-- (5)-- (6)-- (7);
\end{tikzpicture}
\\
$  a_1a_2L(a_3a_4a_5)$&
$  a_1a_2L(a_3)a_4a_5$&
$a_1a_2a_3L(a_4)a_5$\\
\\

\begin{tikzpicture}[scale = \sc, font=\scriptsize\sffamily, thick,main node/.style={circle,inner sep=0.4mm,draw, fill}]
\def\xlb{0}; \def\xub{12}; \def\ylb{0}; \def\yub{2}; \def\buf{0}; \grid
\node[main node] at (0,0) (0) {};
\node[main node] at (2,0) (1) {};
\node[main node] at (4,0) (2) {};
\node[main node] at (6,0) (3) {};
\node[main node] at (8,0) (4) {};
\node[main node] at (9,1) (5) {};
\node[main node] at (11,1) (6) {};
\node[main node] at (12,0) (7) {};
\draw[ultra thick] (0) -- (1) -- (2) -- (3)-- (4)-- (5)-- (6)-- (7);
\end{tikzpicture}
&
\begin{tikzpicture}[scale = \sc, font=\scriptsize\sffamily, thick,main node/.style={circle,inner sep=0.4mm,draw, fill}]
\def\xlb{0}; \def\xub{10}; \def\ylb{0}; \def\yub{3}; \def\buf{0}; \grid
\node[main node] at (0,0) (0) {};
\node[main node] at (1,1) (1) {};
\node[main node] at (2,2) (2) {};
\node[main node] at (4,2) (3) {};
\node[main node] at (6,2) (4) {};
\node[main node] at (8,2) (5) {};
\node[main node] at (9,1) (6) {};
\node[main node] at (10,0) (7) {};
\draw[ultra thick] (0) -- (1) -- (2) -- (3)-- (4)-- (5)-- (6)-- (7);
\end{tikzpicture}
&
\begin{tikzpicture}[scale = \sc, font=\scriptsize\sffamily, thick,main node/.style={circle,inner sep=0.4mm,draw, fill}]
\def\xlb{0}; \def\xub{10}; \def\ylb{0}; \def\yub{3}; \def\buf{0}; \grid
\node[main node] at (0,0) (0) {};
\node[main node] at (1,1) (1) {};
\node[main node] at (2,2) (2) {};
\node[main node] at (4,2) (3) {};
\node[main node] at (5,1) (4) {};
\node[main node] at (7,1) (5) {};
\node[main node] at (9,1) (6) {};
\node[main node] at (10,0) (7) {};
\draw[ultra thick] (0) -- (1) -- (2) -- (3)-- (4)-- (5)-- (6)-- (7);\end{tikzpicture}
\\
$a_1a_2a_3a_4L(a_5)$&
$ L(L(a_1a_2a_3))$&
$ L(L(a_1)a_2a_3)$ \\
\\

\begin{tikzpicture}[scale = \sc, font=\scriptsize\sffamily, thick,main node/.style={circle,inner sep=0.4mm,draw, fill}]
\def\xlb{0}; \def\xub{10}; \def\ylb{0}; \def\yub{3}; \def\buf{0}; \grid
\node[main node] at (0,0) (0) {};
\node[main node] at (1,1) (1) {};
\node[main node] at (2,2) (2) {};
\node[main node] at (4,2) (3) {};
\node[main node] at (5,1) (4) {};
\node[main node] at (6,0) (5) {};
\node[main node] at (8,0) (6) {};
\node[main node] at (10,0) (7) {};
\draw[ultra thick] (0) -- (1) -- (2) -- (3)-- (4)-- (5)-- (6)-- (7);\end{tikzpicture}
&

\begin{tikzpicture}[scale = \sc, font=\scriptsize\sffamily, thick,main node/.style={circle,inner sep=0.4mm,draw, fill}]
\def\xlb{0}; \def\xub{10}; \def\ylb{0}; \def\yub{3}; \def\buf{0}; \grid
\node[main node] at (0,0) (0) {};
\node[main node] at (1,1) (1) {};
\node[main node] at (3,1) (2) {};
\node[main node] at (4,2) (3) {};
\node[main node] at (6,2) (4) {};
\node[main node] at (7,1) (5) {};
\node[main node] at (9,1) (6) {};
\node[main node] at (10,0) (7) {};
\draw[ultra thick] (0) -- (1) -- (2) -- (3)-- (4)-- (5)-- (6)-- (7);
\end{tikzpicture}
&

\begin{tikzpicture}[scale = \sc, font=\scriptsize\sffamily, thick,main node/.style={circle,inner sep=0.4mm,draw, fill}]
\def\xlb{0}; \def\xub{10}; \def\ylb{0}; \def\yub{3}; \def\buf{0}; \grid
\node[main node] at (0,0) (0) {};
\node[main node] at (1,1) (1) {};
\node[main node] at (3,1) (2) {};
\node[main node] at (5,1) (3) {};
\node[main node] at (6,2) (4) {};
\node[main node] at (8,2) (5) {};
\node[main node] at (9,1) (6) {};
\node[main node] at (10,0) (7) {};
\draw[ultra thick] (0) -- (1) -- (2) -- (3)-- (4)-- (5)-- (6)-- (7);\end{tikzpicture}
\\

$ L(L(a_1))a_2a_3 $&
$ L(a_1L(a_2)a_3)$&
$ L(a_1a_2L(a_3)) $\\
  \\
  
\begin{tikzpicture}[scale = \sc, font=\scriptsize\sffamily, thick,main node/.style={circle,inner sep=0.4mm,draw, fill}]
\def\xlb{0}; \def\xub{10}; \def\ylb{0}; \def\yub{3}; \def\buf{0}; \grid
\node[main node] at (0,0) (0) {};
\node[main node] at (1,1) (1) {};
\node[main node] at (3,1) (2) {};
\node[main node] at (4,0) (3) {};
\node[main node] at (5,1) (4) {};
\node[main node] at (7,1) (5) {};
\node[main node] at (8,0) (6) {};
\node[main node] at (10,0) (7) {};
\draw[ultra thick] (0) -- (1) -- (2) -- (3)-- (4)-- (5)-- (6)-- (7);\end{tikzpicture}
&

\begin{tikzpicture}[scale = \sc, font=\scriptsize\sffamily, thick,main node/.style={circle,inner sep=0.4mm,draw, fill}]
\def\xlb{0}; \def\xub{10}; \def\ylb{0}; \def\yub{3}; \def\buf{0}; \grid
\node[main node] at (0,0) (0) {};
\node[main node] at (1,1) (1) {};
\node[main node] at (3,1) (2) {};
\node[main node] at (4,0) (3) {};
\node[main node] at (6,0) (4) {};
\node[main node] at (7,1) (5) {};
\node[main node] at (9,1) (6) {};
\node[main node] at (10,0) (7) {};
\draw[ultra thick] (0) -- (1) -- (2) -- (3)-- (4)-- (5)-- (6)-- (7);\end{tikzpicture}
&

\begin{tikzpicture}[scale = \sc, font=\scriptsize\sffamily, thick,main node/.style={circle,inner sep=0.4mm,draw, fill}]
\def\xlb{0}; \def\xub{10}; \def\ylb{0}; \def\yub{3}; \def\buf{0}; \grid
\node[main node] at (0,0) (0) {};
\node[main node] at (2,0) (1) {};
\node[main node] at (3,1) (2) {};
\node[main node] at (4,2) (3) {};
\node[main node] at (6,2) (4) {};
\node[main node] at (7,1) (5) {};
\node[main node] at (8,0) (6) {};
\node[main node] at (10,0) (7) {};
\draw[ultra thick] (0) -- (1) -- (2) -- (3)-- (4)-- (5)-- (6)-- (7);\end{tikzpicture}
\\
$ L(a_1)L(a_2)a_3 $&
$L(a_1)a_2L(a_3) $&
$a_1L(L(a_2))a_3 $\\
\\

\begin{tikzpicture}[scale = \sc, font=\scriptsize\sffamily, thick,main node/.style={circle,inner sep=0.4mm,draw, fill}]
\def\xlb{0}; \def\xub{10}; \def\ylb{0}; \def\yub{3}; \def\buf{0}; \grid
\node[main node] at (0,0) (0) {};
\node[main node] at (2,0) (1) {};
\node[main node] at (3,1) (2) {};
\node[main node] at (5,1) (3) {};
\node[main node] at (6,0) (4) {};
\node[main node] at (7,1) (5) {};
\node[main node] at (9,1) (6) {};
\node[main node] at (10,0) (7) {};
\draw[ultra thick] (0) -- (1) -- (2) -- (3)-- (4)-- (5)-- (6)-- (7);\end{tikzpicture}
&

\begin{tikzpicture}[scale = \sc, font=\scriptsize\sffamily, thick,main node/.style={circle,inner sep=0.4mm,draw, fill}]
\def\xlb{0}; \def\xub{10}; \def\ylb{0}; \def\yub{3}; \def\buf{0}; \grid
\node[main node] at (0,0) (0) {};
\node[main node] at (2,0) (1) {};
\node[main node] at (4,0) (2) {};
\node[main node] at (5,1) (3) {};
\node[main node] at (6,2) (4) {};
\node[main node] at (8,2) (5) {};
\node[main node] at (9,1) (6) {};
\node[main node] at (10,0) (7) {};
\draw[ultra thick] (0) -- (1) -- (2) -- (3)-- (4)-- (5)-- (6)-- (7);\end{tikzpicture}
&

\begin{tikzpicture}[scale = \sc, font=\scriptsize\sffamily, thick,main node/.style={circle,inner sep=0.4mm,draw, fill}]
\def\xlb{0}; \def\xub{8}; \def\ylb{0}; \def\yub{4}; \def\buf{0}; \grid
\node[main node] at (0,0) (0) {};
\node[main node] at (1,1) (1) {};
\node[main node] at (2,2) (2) {};
\node[main node] at (3,3) (3) {};
\node[main node] at (5,3) (4) {};
\node[main node] at (6,2) (5) {};
\node[main node] at (7,1) (6) {};
\node[main node] at (8,0) (7) {};
\draw[ultra thick] (0) -- (1) -- (2) -- (3)-- (4)-- (5)-- (6)-- (7);
\end{tikzpicture}
\\
$a_1L(a_2)L(a_3) $&
$ a_1a_2L(L(a_3)) $&
$L(L(L(a_1)))$
\end{tabular}
\caption{Illustrating the bijection $f_1$ for all $M \in \M_3$ with $\topt(M)=3$}\label{fig301}
\end{figure}
Then we have the following.

\begin{theorem}\label{thm301}
Let \( d \geq 2 \) and \( n, k \geq 0 \) be integers. Then \( N_d(n,k) \) counts the number of Schr\"{o}der paths in \( \mathcal{S}_d \) with semilength \( (n-k)(d-1) + k+1 \) and exactly \( k \) up steps.
\end{theorem}

\begin{proof}
Given \( M \in \mathcal{M}_d \), observe that \( f_1(M) \) is indeed a Schr\"{o}der path --- it contains only \( U \), \( H \), and \( D \) steps with an equal number of \( U \) and \( D \) steps, and it cannot contain a prefix with more \( D \) steps than \( U \) steps. Additionally, every monomial has a positive degree that is congruent to \( 1 \) modulo $d-1$, which implies that the total number of horizontal steps in \( f_1(M) \) is equal to \( 1 + j(d-1) \) for some integer \( j \geq 0 \). 

Furthermore, since the argument of an instance of \( L \) in \( M \) is a monomial in its own right, and the opening and closing parentheses of \( L \) are mapped to a pair of matching up and down steps in \( f_1(M) \), we see that the number of horizontal steps between them must also be equal to \( 1 + j(d-1) \) for some \( j \geq 0 \). Hence, we conclude that \( f_1(M) \in \mathcal{S}_d \).

Next, observe that \( |f_1(M)| = \deg(M) + \lopt(M) \), and there are exactly \( \lopt(M) \) up steps in \( f_1(M) \). The mapping \( f_1 \) is also clearly invertible. Thus, our claim follows.
\end{proof}

Given a Schr\"{o}der path, we call an instance of \( UD \) (i.e., an up step immediately followed by a down step) a \emph{peak} in the path. Notice that \( \mathcal{S}_2 \) is simply the set of nonempty Schr\"{o}der paths with no peaks, which is known to be counted by the Catalan numbers~\cite[Chapter 2, Exercise 45]{Stanley15}.

\subsection{Ordered trees with restricted outdegrees}

We next show that \( N_d(n,k) \) also counts a family of ordered trees  (i.e., trees where the order of children at each node matters). Given an ordered tree \( T \), the \emph{outdegree} of a vertex \( v \in T \) is the number of children of \( v \) in \( T \). A vertex is considered a \emph{leaf} if it has an outdegree of zero; otherwise, it is classified as an \emph{internal node}. 

Next, for an integer \( d \geq 2 \), let \( \mathcal{T}_d \) denote the set of unlabelled ordered trees that contain at least one internal node, where the outdegree of every internal node is congruent to \( 1 \) modulo $d-1$. For instance, Figure~\ref{fig302} displays the 12 elements in \( \mathcal{T}_3 \) with 5 edges.

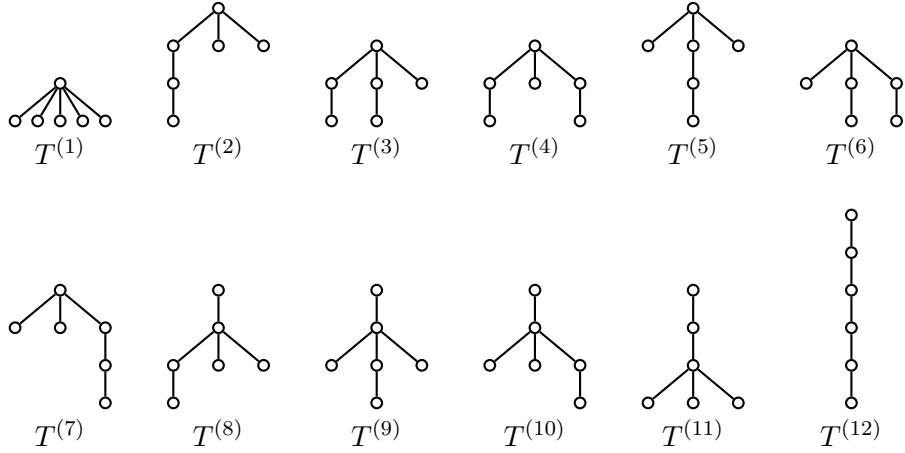
\begin{figure}[htbp]
\centering
\def\sc{0.3}
\def\ysc{0.5}
\begin{tabular}{cccccc}

\begin{tikzpicture}[xscale=\sc, yscale = \ysc,thick,main node/.style={circle,inner sep=0.5mm,draw,font=\small\sffamily}]
\node[main node] at (2,2) (0) {};
\node[main node] at (0,1) (1) {};
\node[main node] at (1,1) (2) {};
\node[main node] at (2,1) (3) {};
\node[main node] at (3,1) (4) {};
\node[main node] at (4,1) (5) {};

 \path[every node/.style={font=\sffamily}]
(0) edge (1) (0) edge (2) (0) edge (3) (0) edge (4) (0) edge (5);
\end{tikzpicture}

&\quad

\begin{tikzpicture}[xscale=\sc, yscale = \ysc,thick,main node/.style={circle,inner sep=0.5mm,draw,font=\small\sffamily}]
\node[main node] at (2,2) (0) {};
\node[main node] at (0,1) (1) {};
\node[main node] at (2,1) (2) {};
\node[main node] at (4,1) (3) {};
\node[main node] at (0,0) (4) {};
\node[main node] at (0,-1) (5) {};

 \path[every node/.style={font=\sffamily}]
(0) edge (1) (0) edge (2) (0) edge (3) (1) edge (4) (4) edge (5);
\end{tikzpicture}

&\quad
\begin{tikzpicture}[xscale=\sc, yscale = \ysc,thick,main node/.style={circle,inner sep=0.5mm,draw,font=\small\sffamily}]
\node[main node] at (2,2) (0) {};
\node[main node] at (0,1) (1) {};
\node[main node] at (2,1) (2) {};
\node[main node] at (4,1) (3) {};
\node[main node] at (0,0) (4) {};
\node[main node] at (2,0) (5) {};

 \path[every node/.style={font=\sffamily}]
(0) edge (1) (0) edge (2) (0) edge (3) (1) edge (4) (2) edge (5);
\end{tikzpicture}

&\quad
\begin{tikzpicture}[xscale=\sc, yscale = \ysc,thick,main node/.style={circle,inner sep=0.5mm,draw,font=\small\sffamily}]
\node[main node] at (2,2) (0) {};
\node[main node] at (0,1) (1) {};
\node[main node] at (2,1) (2) {};
\node[main node] at (4,1) (3) {};
\node[main node] at (0,0) (4) {};
\node[main node] at (4,0) (5) {};

 \path[every node/.style={font=\sffamily}]
(0) edge (1) (0) edge (2) (0) edge (3) (1) edge (4) (3) edge (5);
\end{tikzpicture}

&\quad
\begin{tikzpicture}[xscale=\sc, yscale = \ysc,thick,main node/.style={circle,inner sep=0.5mm,draw,font=\small\sffamily}]
\node[main node] at (2,2) (0) {};
\node[main node] at (0,1) (1) {};
\node[main node] at (2,1) (2) {};
\node[main node] at (4,1) (3) {};
\node[main node] at (2,0) (4) {};
\node[main node] at (2,-1) (5) {};

 \path[every node/.style={font=\sffamily}]
(0) edge (1) (0) edge (2) (0) edge (3)  (2) edge (4) (4) edge (5);
\end{tikzpicture}

&\quad
\begin{tikzpicture}[xscale=\sc, yscale = \ysc,thick,main node/.style={circle,inner sep=0.5mm,draw,font=\small\sffamily}]
\node[main node] at (2,2) (0) {};
\node[main node] at (0,1) (1) {};
\node[main node] at (2,1) (2) {};
\node[main node] at (4,1) (3) {};
\node[main node] at (2,0) (4) {};
\node[main node] at (4,0) (5) {};

 \path[every node/.style={font=\sffamily}]
(0) edge (1) (0) edge (2) (0) edge (3) (2) edge (4) (3) edge (5);
\end{tikzpicture}
\\

$T^{(1)}$ &\quad
$T^{(2)}$ &\quad
$T^{(3)}$ &\quad
$T^{(4)}$ &\quad
$T^{(5)}$ &\quad
$T^{(6)}$ 
\\
\\

\begin{tikzpicture}[xscale=\sc, yscale = \ysc,thick,main node/.style={circle,inner sep=0.5mm,draw,font=\small\sffamily}]
\node[main node] at (2,2) (0) {};
\node[main node] at (0,1) (1) {};
\node[main node] at (2,1) (2) {};
\node[main node] at (4,1) (3) {};
\node[main node] at (4,0) (4) {};
\node[main node] at (4,-1) (5) {};

 \path[every node/.style={font=\sffamily}]
(0) edge (1) (0) edge (2) (0) edge (3) (3) edge (4) (4) edge (5);
\end{tikzpicture}

&\quad
\begin{tikzpicture}[xscale=\sc, yscale = \ysc,thick,main node/.style={circle,inner sep=0.5mm,draw,font=\small\sffamily}]
\node[main node] at (2,2) (0) {};
\node[main node] at (2,1) (1) {};
\node[main node] at (0,0) (2) {};
\node[main node] at (2,0) (3) {};
\node[main node] at (4,0) (4) {};
\node[main node] at (0,-1) (5) {};

 \path[every node/.style={font=\sffamily}]
(0) edge (1) (1) edge (2) (1) edge (3) (1) edge (4) (2) edge (5);
\end{tikzpicture}

&\quad
\begin{tikzpicture}[xscale=\sc, yscale = \ysc,thick,main node/.style={circle,inner sep=0.5mm,draw,font=\small\sffamily}]
\node[main node] at (2,2) (0) {};
\node[main node] at (2,1) (1) {};
\node[main node] at (0,0) (2) {};
\node[main node] at (2,0) (3) {};
\node[main node] at (4,0) (4) {};
\node[main node] at (2,-1) (5) {};

 \path[every node/.style={font=\sffamily}]
(0) edge (1) (1) edge (2) (1) edge (3) (1) edge (4) (3) edge (5);
\end{tikzpicture}

&\quad
\begin{tikzpicture}[xscale=\sc, yscale = \ysc,thick,main node/.style={circle,inner sep=0.5mm,draw,font=\small\sffamily}]
\node[main node] at (2,2) (0) {};
\node[main node] at (2,1) (1) {};
\node[main node] at (0,0) (2) {};
\node[main node] at (2,0) (3) {};
\node[main node] at (4,0) (4) {};
\node[main node] at (4,-1) (5) {};

 \path[every node/.style={font=\sffamily}]
(0) edge (1) (1) edge (2) (1) edge (3) (1) edge (4) (4) edge (5);
\end{tikzpicture}

&\quad
\begin{tikzpicture}[xscale=\sc, yscale = \ysc,thick,main node/.style={circle,inner sep=0.5mm,draw,font=\small\sffamily}]
\node[main node] at (2,2) (0) {};
\node[main node] at (2,1) (1) {};
\node[main node] at (2,0) (2) {};
\node[main node] at (0,-1) (3) {};
\node[main node] at (2,-1) (4) {};
\node[main node] at (4,-1) (5) {};

 \path[every node/.style={font=\sffamily}]
(0) edge (1) (1) edge (2) (2) edge (3) (2) edge (4) (2) edge (5);
\end{tikzpicture}

&\quad
\begin{tikzpicture}[xscale=\sc, yscale = \ysc,thick,main node/.style={circle,inner sep=0.5mm,draw,font=\small\sffamily}]
\node[main node] at (2,2) (0) {};
\node[main node] at (2,1) (1) {};
\node[main node] at (2,0) (2) {};
\node[main node] at (2,-1) (3) {};
\node[main node] at (2,-2) (4) {};
\node[main node] at (2,-3) (5) {};

 \path[every node/.style={font=\sffamily}]
(0) edge (1) (1) edge (2) (2) edge (3) (3) edge (4) (4) edge (5);
\end{tikzpicture}
\\
$T^{(7)}$ &\quad
$T^{(8)}$ &\quad
$T^{(9)}$ &\quad
$T^{(10)}$ &\quad
$T^{(11)}$ &\quad
$T^{(12)}$ 

\end{tabular}
\caption{The 12 trees in \( \T_3\) with 5 edges}
\label{fig302}
\end{figure}

We next define a bijection $f_2 : \M_d \to \T_d$, as follows. Given $M \in \M_d$, we write $M = M_1 \cdots M_{\ell}$, where $M_i \in \overline{\M}_d$ is irreducible for each $i \in [\ell]$.  Then $f_2(M)$ is defined recursively: The root node has outdegree $\ell$, with subtrees $T_1, \ldots, T_{\ell}$ arranged from left to right.  For each $i \in [\ell]$, if $M_i$ is an indeterminate, then $T_i$ is defined to be a one-node tree; otherwise, $M_i = L(M_i')$ for some $M_i' \in \M_d$, and we define $T_i \ce f_2(M_i’)$.

For example, Figure~\ref{fig303} illustrates the mapping $f_2$ for all $21$ ternary operator monomials $M \in \M_3$ with $\topt(M) = 3$. 

\begin{figure}[htbp]
\centering
\scriptsize
\def\sc{0.3}
\def\ysc{0.5}
{
\setlength{\tabcolsep}{3pt}
\begin{tabular}{ccccccc}

\begin{tikzpicture}[xscale=\sc, yscale = \ysc,thick,main node/.style={circle,inner sep=0.5mm,draw,font=\small\sffamily}]
\node[main node] at (3,2) (0) {};
\node[main node] at (0,1) (1) {};
\node[main node] at (1,1) (2) {};
\node[main node] at (2,1) (3) {};
\node[main node] at (3,1) (4) {};
\node[main node] at (4,1) (5) {};
\node[main node] at (5,1) (6) {};
\node[main node] at (6,1) (7) {};

 \path[every node/.style={font=\sffamily}]
(0) edge (1) (0) edge (2) (0) edge (3) (0) edge (4) (0) edge (5) (0) edge (6) (0) edge (7);
\end{tikzpicture}
&
\begin{tikzpicture}[xscale=\sc, yscale = \ysc,thick,main node/.style={circle,inner sep=0.5mm,draw,font=\small\sffamily}]
\node[main node] at (2,2) (0) {};
\node[main node] at (2,1) (1) {};
\node[main node] at (0,0) (2) {};
\node[main node] at (1,0) (3) {};
\node[main node] at (2,0) (4) {};
\node[main node] at (3,0) (5) {};
\node[main node] at (4,0) (6) {};

 \path[every node/.style={font=\sffamily}]
(0) edge (1) (1) edge (2) (1) edge (3) (1) edge (4) (1) edge (5) (1) edge (6);
\end{tikzpicture}

&
\begin{tikzpicture}[xscale=\sc, yscale = \ysc,thick,main node/.style={circle,inner sep=0.5mm,draw,font=\small\sffamily}]
\node[main node] at (2,2) (0) {};
\node[main node] at (0,1) (1) {};
\node[main node] at (2,1) (2) {};
\node[main node] at (4,1) (3) {};
\node[main node] at (-1,0) (4) {};
\node[main node] at (0,0) (5) {};
\node[main node] at (1,0) (6) {};

 \path[every node/.style={font=\sffamily}]
(0) edge (1) (0) edge (2) (0) edge (3) (1) edge (4) (1) edge (5) (1) edge (6);
\end{tikzpicture}
&
\begin{tikzpicture}[xscale=\sc, yscale = \ysc,thick,main node/.style={circle,inner sep=0.5mm,draw,font=\small\sffamily}]
\node[main node] at (2,2) (0) {};
\node[main node] at (0,1) (1) {};
\node[main node] at (1,1) (2) {};
\node[main node] at (2,1) (3) {};
\node[main node] at (3,1) (4) {};
\node[main node] at (4,1) (5) {};
\node[main node] at (0,0) (6) {};

 \path[every node/.style={font=\sffamily}]
(0) edge (1) (0) edge (2) (0) edge (3) (0) edge (4) (0) edge (5) (1) edge (6);
\end{tikzpicture}
&
\begin{tikzpicture}[xscale=\sc, yscale = \ysc,thick,main node/.style={circle,inner sep=0.5mm,draw,font=\small\sffamily}]
\node[main node] at (2,2) (0) {};
\node[main node] at (0,1) (1) {};
\node[main node] at (2,1) (2) {};
\node[main node] at (4,1) (3) {};
\node[main node] at (1,0) (4) {};
\node[main node] at (2,0) (5) {};
\node[main node] at (3,0) (6) {};

 \path[every node/.style={font=\sffamily}]
(0) edge (1) (0) edge (2) (0) edge (3) (2) edge (4) (2) edge (5) (2) edge (6);
\end{tikzpicture}
&
\begin{tikzpicture}[xscale=\sc, yscale = \ysc,thick,main node/.style={circle,inner sep=0.5mm,draw,font=\small\sffamily}]
\node[main node] at (2,2) (0) {};
\node[main node] at (0,1) (1) {};
\node[main node] at (1,1) (2) {};
\node[main node] at (2,1) (3) {};
\node[main node] at (3,1) (4) {};
\node[main node] at (4,1) (5) {};
\node[main node] at (1,0) (6) {};

 \path[every node/.style={font=\sffamily}]
(0) edge (1) (0) edge (2) (0) edge (3) (0) edge (4) (0) edge (5) (2) edge (6);
\end{tikzpicture}
&
\begin{tikzpicture}[xscale=\sc, yscale = \ysc,thick,main node/.style={circle,inner sep=0.5mm,draw,font=\small\sffamily}]
\node[main node] at (2,2) (0) {};
\node[main node] at (0,1) (1) {};
\node[main node] at (2,1) (2) {};
\node[main node] at (4,1) (3) {};
\node[main node] at (3,0) (4) {};
\node[main node] at (4,0) (5) {};
\node[main node] at (5,0) (6) {};

 \path[every node/.style={font=\sffamily}]
(0) edge (1) (0) edge (2) (0) edge (3) (3) edge (4) (3) edge (5) (3) edge (6);
\end{tikzpicture}
\\

$a_1a_2a_3a_4a_5a_6a_7$ &
$L(a_1a_2a_3a_4a_5)$ &
$L(a_1a_2a_3)a_4a_5$&
$L(a_1)a_2a_3a_4a_5$&
$a_1L(a_2a_3a_4)a_5$&
$ a_1L(a_2)a_3a_4a_5$&
$  a_1a_2L(a_3a_4a_5)$\\
 \\
  
\begin{tikzpicture}[xscale=\sc, yscale = \ysc,thick,main node/.style={circle,inner sep=0.5mm,draw,font=\small\sffamily}]
\node[main node] at (2,2) (0) {};
\node[main node] at (0,1) (1) {};
\node[main node] at (1,1) (2) {};
\node[main node] at (2,1) (3) {};
\node[main node] at (3,1) (4) {};
\node[main node] at (4,1) (5) {};
\node[main node] at (2,0) (6) {};

 \path[every node/.style={font=\sffamily}]
(0) edge (1) (0) edge (2) (0) edge (3) (0) edge (4) (0) edge (5) (3) edge (6);
\end{tikzpicture}
&
\begin{tikzpicture}[xscale=\sc, yscale = \ysc,thick,main node/.style={circle,inner sep=0.5mm,draw,font=\small\sffamily}]
\node[main node] at (2,2) (0) {};
\node[main node] at (0,1) (1) {};
\node[main node] at (1,1) (2) {};
\node[main node] at (2,1) (3) {};
\node[main node] at (3,1) (4) {};
\node[main node] at (4,1) (5) {};
\node[main node] at (3,0) (6) {};

 \path[every node/.style={font=\sffamily}]
(0) edge (1) (0) edge (2) (0) edge (3) (0) edge (4) (0) edge (5) (4) edge (6);
\end{tikzpicture}
&
\begin{tikzpicture}[xscale=\sc, yscale = \ysc,thick,main node/.style={circle,inner sep=0.5mm,draw,font=\small\sffamily}]
\node[main node] at (2,2) (0) {};
\node[main node] at (0,1) (1) {};
\node[main node] at (1,1) (2) {};
\node[main node] at (2,1) (3) {};
\node[main node] at (3,1) (4) {};
\node[main node] at (4,1) (5) {};
\node[main node] at (4,0) (6) {};

 \path[every node/.style={font=\sffamily}]
(0) edge (1) (0) edge (2) (0) edge (3) (0) edge (4) (0) edge (5) (5) edge (6);
\end{tikzpicture}
&
\begin{tikzpicture}[xscale=\sc, yscale = \ysc,thick,main node/.style={circle,inner sep=0.5mm,draw,font=\small\sffamily}]
\node[main node] at (2,2) (0) {};
\node[main node] at (2,1) (1) {};
\node[main node] at (2,0) (2) {};
\node[main node] at (0,-1) (3) {};
\node[main node] at (2,-1) (4) {};
\node[main node] at (4,-1) (5) {};

 \path[every node/.style={font=\sffamily}]
(0) edge (1) (1) edge (2) (2) edge (3) (2) edge (4) (2) edge (5);
\end{tikzpicture}
&

\begin{tikzpicture}[xscale=\sc, yscale = \ysc,thick,main node/.style={circle,inner sep=0.5mm,draw,font=\small\sffamily}]
\node[main node] at (2,2) (0) {};
\node[main node] at (2,1) (1) {};
\node[main node] at (0,0) (2) {};
\node[main node] at (2,0) (3) {};
\node[main node] at (4,0) (4) {};
\node[main node] at (0,-1) (5) {};

 \path[every node/.style={font=\sffamily}]
(0) edge (1) (1) edge (2) (1) edge (3) (1) edge (4) (2) edge (5);
\end{tikzpicture}

&

\begin{tikzpicture}[xscale=\sc, yscale = \ysc,thick,main node/.style={circle,inner sep=0.5mm,draw,font=\small\sffamily}]
\node[main node] at (2,2) (0) {};
\node[main node] at (0,1) (1) {};
\node[main node] at (2,1) (2) {};
\node[main node] at (4,1) (3) {};
\node[main node] at (0,0) (4) {};
\node[main node] at (0,-1) (5) {};

 \path[every node/.style={font=\sffamily}]
(0) edge (1) (0) edge (2) (0) edge (3) (1) edge (4) (4) edge (5);
\end{tikzpicture}
&
    \begin{tikzpicture}[xscale=\sc, yscale = \ysc,thick,main node/.style={circle,inner sep=0.5mm,draw,font=\small\sffamily}]
\node[main node] at (2,2) (0) {};
\node[main node] at (2,1) (1) {};
\node[main node] at (0,0) (2) {};
\node[main node] at (2,0) (3) {};
\node[main node] at (4,0) (4) {};
\node[main node] at (2,-1) (5) {};

 \path[every node/.style={font=\sffamily}]
(0) edge (1) (1) edge (2) (1) edge (3) (1) edge (4) (3) edge (5);
\end{tikzpicture}

\\
$  a_1a_2L(a_3)a_4a_5$&
$a_1a_2a_3L(a_4)a_5$&
$a_1a_2a_3a_4L(a_5)$&
$ L(L(a_1a_2a_3))$&
$ L(L(a_1)a_2a_3)$ &
$ L(L(a_1))a_2a_3 $&
$ L(a_1L(a_2)a_3)$\\
  \\
  
\begin{tikzpicture}[xscale=\sc, yscale = \ysc,thick,main node/.style={circle,inner sep=0.5mm,draw,font=\small\sffamily}]
\node[main node] at (2,2) (0) {};
\node[main node] at (2,1) (1) {};
\node[main node] at (0,0) (2) {};
\node[main node] at (2,0) (3) {};
\node[main node] at (4,0) (4) {};
\node[main node] at (4,-1) (5) {};

 \path[every node/.style={font=\sffamily}]
(0) edge (1) (1) edge (2) (1) edge (3) (1) edge (4) (4) edge (5);
\end{tikzpicture}

&
 \begin{tikzpicture}[xscale=\sc, yscale = \ysc,thick,main node/.style={circle,inner sep=0.5mm,draw,font=\small\sffamily}]
\node[main node] at (2,2) (0) {};
\node[main node] at (0,1) (1) {};
\node[main node] at (2,1) (2) {};
\node[main node] at (4,1) (3) {};
\node[main node] at (0,0) (4) {};
\node[main node] at (2,0) (5) {};

 \path[every node/.style={font=\sffamily}]
(0) edge (1) (0) edge (2) (0) edge (3) (1) edge (4) (2) edge (5);
\end{tikzpicture}

&

 \begin{tikzpicture}[xscale=\sc, yscale = \ysc,thick,main node/.style={circle,inner sep=0.5mm,draw,font=\small\sffamily}]
\node[main node] at (2,2) (0) {};
\node[main node] at (0,1) (1) {};
\node[main node] at (2,1) (2) {};
\node[main node] at (4,1) (3) {};
\node[main node] at (0,0) (4) {};
\node[main node] at (4,0) (5) {};

 \path[every node/.style={font=\sffamily}]
(0) edge (1) (0) edge (2) (0) edge (3) (1) edge (4) (3) edge (5);
\end{tikzpicture}

&

 \begin{tikzpicture}[xscale=\sc, yscale = \ysc,thick,main node/.style={circle,inner sep=0.5mm,draw,font=\small\sffamily}]
\node[main node] at (2,2) (0) {};
\node[main node] at (0,1) (1) {};
\node[main node] at (2,1) (2) {};
\node[main node] at (4,1) (3) {};
\node[main node] at (2,0) (4) {};
\node[main node] at (2,-1) (5) {};

 \path[every node/.style={font=\sffamily}]
(0) edge (1) (0) edge (2) (0) edge (3)  (2) edge (4) (4) edge (5);
\end{tikzpicture}
&

 \begin{tikzpicture}[xscale=\sc, yscale = \ysc,thick,main node/.style={circle,inner sep=0.5mm,draw,font=\small\sffamily}]
\node[main node] at (2,2) (0) {};
\node[main node] at (0,1) (1) {};
\node[main node] at (2,1) (2) {};
\node[main node] at (4,1) (3) {};
\node[main node] at (2,0) (4) {};
\node[main node] at (4,0) (5) {};

 \path[every node/.style={font=\sffamily}]
(0) edge (1) (0) edge (2) (0) edge (3) (2) edge (4) (3) edge (5);
\end{tikzpicture}
&
 \begin{tikzpicture}[xscale=\sc, yscale = \ysc,thick,main node/.style={circle,inner sep=0.5mm,draw,font=\small\sffamily}]
\node[main node] at (2,2) (0) {};
\node[main node] at (0,1) (1) {};
\node[main node] at (2,1) (2) {};
\node[main node] at (4,1) (3) {};
\node[main node] at (4,0) (4) {};
\node[main node] at (4,-1) (5) {};

 \path[every node/.style={font=\sffamily}]
(0) edge (1) (0) edge (2) (0) edge (3) (3) edge (4) (4) edge (5);
\end{tikzpicture}

&
 \begin{tikzpicture}[xscale=\sc, yscale = \ysc,thick,main node/.style={circle,inner sep=0.5mm,draw,font=\small\sffamily}]
\node[main node] at (2,2) (0) {};
\node[main node] at (2,1) (1) {};
\node[main node] at (2,0) (2) {};
\node[main node] at (2,-1) (3) {};
\node[main node] at (2,-2) (4) {};

 \path[every node/.style={font=\sffamily}]
(0) edge (1) (1) edge (2) (2) edge (3) (3) edge (4);
\end{tikzpicture}
\\
$ L(a_1a_2L(a_3)) $&
$ L(a_1)L(a_2)a_3 $&
$L(a_1)a_2L(a_3) $&
$a_1L(L(a_2))a_3 $&
$a_1L(a_2)L(a_3) $&
$ a_1a_2L(L(a_3)) $&
$L(L(L(a_1)))$
\end{tabular}
}
\caption{Illustrating the bijection $f_2$ for all $M \in \M_3$ with $\topt(M)=3$}
\label{fig303}
\end{figure}
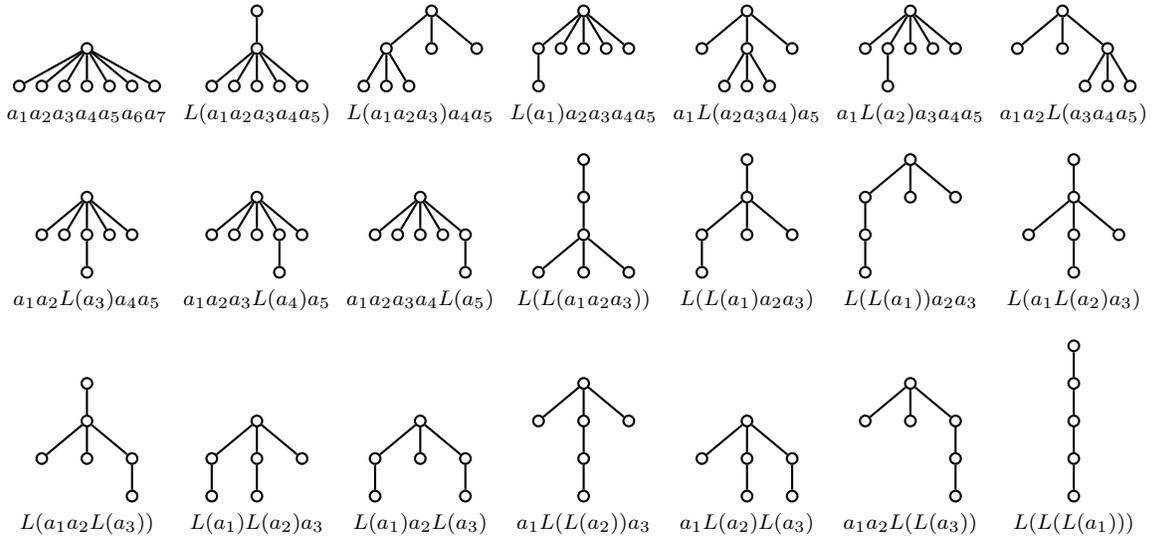

Then we have the following.

\begin{theorem}\label{thm302}
Let $d \geq 2$ and $n, k \geq 0$ be integers. Then $N_d(n, k)$ counts the number of trees in $\T_d$ with $(n-k)(d-1) + k+1$ edges and $k+1$ internal nodes.
\end{theorem}

\begin{proof}
Observe that when $M \in \M_d$ is written as a product of $\ell$ irreducible operator monomials, we have that $\ell$ is congruent to $1$ modulo $d-1$. Thus, every internal node in $f_2(M)$ has outdegree congruent to $1$ modulo $d-1$, and hence $f_2(M) \in \T_d$.  Next, each non-root internal node in $f_2(M)$ corresponds to an application of $L$ in $M$, and thus the number of internal nodes in $f_2(M)$ is $\lopt(M)+1$. Moreover, each leaf of $f_2(M)$ corresponds to an indeterminate in $M$, which implies that $f_2(M)$ has $\deg(M) + \lopt(M) + 1$ vertices, and thus $\deg(M) + \lopt(M)$ edges.

Next, we describe the inverse mapping $f_2^{-1}$. Given $T \in \T_d$, the root of $T$ is an internal node by definition of $\T_d$. Let $T_1, \ldots, T_{\ell}$ be the subtrees of the root node of $T$ from left to right.  Define the monomial $f_2^{-1}(T)$ as $M_1 \cdots M_{\ell}$ such that $M_i$ is an indeterminate if $T_i$ consists of a single vertex, and $M_i = L(f_2^{-1}(T_i))$ otherwise.  Therefore, $f_2 : \M_d \to \T_d$ is indeed a bijection, and the claim follows.
\end{proof}

We remark that there is a more visual and non-recursive way of describing the inverse mapping $f_2^{-1}$. Given a tree $T \in \T_d$, write symbols $a_1, a_2, \ldots$ under each leaf of $T$ from left to right. Then, for each non-root internal node of $T$, write $L($ to the left of the node and $)$ to the right of the node. Finally, one can write out the monomial $f_2^{-1}(T)$ by starting at the top of the root node of $T$ and then ``walking'' counterclockwise around the tree and picking up the written symbols in order. For example, for the tree in $T \in \T_3$ given in Figure~\ref{fig304}, we have
\[
f_2^{-1}(T) = L(L(a_1)a_2L(a_3a_4a_5)).
\]

\begin{figure}[htbp]
\centering
\begin{tikzpicture}[xscale=1.4, yscale = 1.2 ,thick,main node/.style={circle,inner sep=0.5mm,draw,font=\small\sffamily}]

\node[main node] at (2,5) (8) {};
\node[main node] at (2,4) (0) {};
\node[main node] at (1,3) (1) {};
\node[main node] at (2,3) (2) {};
\node[main node] at (3,3) (3) {};
\node[main node] at (1,2) (4) {};
\node[main node] at (2.3,2) (5) {};
\node[main node] at (3,2) (6) {};
\node[main node] at (3.7,2) (7) {};

\node[label=below:{$a_1$}] at ( 1,2){};
\node[label=below:{$a_2$}] at ( 2,3){};
\node[label=below:{$a_3$}] at ( 2.3,2){};
\node[label=below:{$a_4$}] at ( 3,2){};
\node[label=below:{$a_5$}] at ( 3.7,2){};
\node[label=left:{$L($}] at (2,4){};
\node[label=right:{$)$}] at (2,4){};
\node[label=left:{$L($}] at (1,3){};
\node[label=right:{$)$}] at (1,3){};
\node[label=left:{$L($}] at (3,3){};
\node[label=right:{$)$}] at (3,3){};

 \path[every node/.style={font=\sffamily}]
(8) edge (0) (0) edge (1) (0) edge (2) (0) edge (3) (1) edge (4) (3) edge (5) (3) edge (6) (3) edge (7);
\end{tikzpicture}
\caption{Illustrating the inverse mapping $f_2^{-1}$}\label{fig304}
\end{figure}
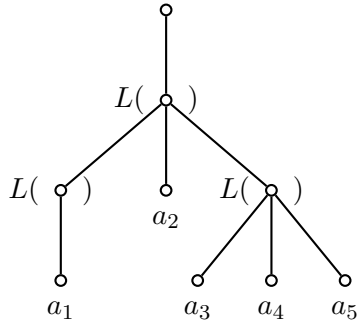

Notice that when $d=2$, $\T_d$ is simply the set of all nonempty ordered trees. Thus, Theorem~\ref{thm302} specializes to the fact that the Narayana number $N_2(n,k)$ counts the number of ordered trees with $n+1$ edges and $k+1$ internal nodes (see, for instance, \seqnum{A001263}).

\subsection{Dyck paths with restricted ascents}\label{sec303}

Now that we have a combinatorial interpretation of $N_d(n,k)$ in terms of a particular subset of ordered trees, we can leverage known bijections between classic combinatorial objects to obtain other interpretations of $N_d(n,k)$. For the remainder of the section, we will detail two such interpretations.

First, a \emph{Dyck path} is a Schr\"{o}der path which does not contain any horizontal steps (i.e., it uses only up and down steps). Given an integer $d \geq 2$, let $\Q_d$ denote the set of nonempty Dyck paths in which every \emph{ascent} (i.e., a maximal subsequence of up steps in the path) has length congruent to $1$ modulo $d-1$. For example, Figure~\ref{fig305} lists the $12$ paths in $\Q_3$ with semilength $5$.
\begin{figure}[htbp]
\centering
\def\sc{0.3}
\def\grid{\foreach \x in {\xlb ,...,\xub}
{    \ifthenelse{\NOT 0 = \x}{\draw[thick](\x ,-1pt) -- (\x ,1pt);}{}
\draw[dotted, thick](\x,\ylb- \buf) -- (\x,\yub + \buf);}
\foreach \y in {\ylb ,...,\yub}
{    \ifthenelse{\NOT 0 = \y}{\draw[thick](-1pt, \y) -- (1pt, \y);}{}
\draw[dotted, thick](\xlb- \buf, \y) -- (\xub + \buf, \y);}
}
\begin{tabular}{cccccc}
\begin{tikzpicture}[scale = \sc, font=\scriptsize\sffamily, thick,main node/.style={circle,inner sep=0.4mm,draw, fill}]
\def\xlb{0}; \def\xub{10}; \def\ylb{0}; \def\yub{5}; \def\buf{0}; \grid
\node[main node] at (0,0) (0) {};
\node[main node] at (1,1) (1) {};
\node[main node] at (2,2) (2) {};
\node[main node] at (3,3) (3) {};
\node[main node] at (4,4) (4) {};
\node[main node] at (5,5) (5) {};
\node[main node] at (6,4) (6) {};
\node[main node] at (7,3) (7) {};
\node[main node] at (8,2) (8) {};
\node[main node] at (9,1) (9) {};
\node[main node] at (10,0) (10) {};\draw[ultra thick] (0) -- (1) -- (2) -- (3)-- (4)-- (5)-- (6)-- (7)-- (8)-- (9)-- (10);
\end{tikzpicture}
&

\begin{tikzpicture}[scale = \sc, font=\scriptsize\sffamily, thick,main node/.style={circle,inner sep=0.4mm,draw, fill}]
\def\xlb{0}; \def\xub{10}; \def\ylb{0}; \def\yub{5}; \def\buf{0}; \grid
\node[main node] at (0,0) (0) {};
\node[main node] at (1,1) (1) {};
\node[main node] at (2,2) (2) {};
\node[main node] at (3,3) (3) {};
\node[main node] at (4,2) (4) {};
\node[main node] at (5,3) (5) {};
\node[main node] at (6,2) (6) {};
\node[main node] at (7,3) (7) {};
\node[main node] at (8,2) (8) {};
\node[main node] at (9,1) (9) {};
\node[main node] at (10,0) (10) {};\draw[ultra thick] (0) -- (1) -- (2) -- (3)-- (4)-- (5)-- (6)-- (7)-- (8)-- (9)-- (10);
\end{tikzpicture}
&

\begin{tikzpicture}[scale = \sc, font=\scriptsize\sffamily, thick,main node/.style={circle,inner sep=0.4mm,draw, fill}]
\def\xlb{0}; \def\xub{10}; \def\ylb{0}; \def\yub{5}; \def\buf{0}; \grid
\node[main node] at (0,0) (0) {};
\node[main node] at (1,1) (1) {};
\node[main node] at (2,2) (2) {};
\node[main node] at (3,3) (3) {};
\node[main node] at (4,2) (4) {};
\node[main node] at (5,3) (5) {};
\node[main node] at (6,2) (6) {};
\node[main node] at (7,1) (7) {};
\node[main node] at (8,2) (8) {};
\node[main node] at (9,1) (9) {};
\node[main node] at (10,0) (10) {};\draw[ultra thick] (0) -- (1) -- (2) -- (3)-- (4)-- (5)-- (6)-- (7)-- (8)-- (9)-- (10);
\end{tikzpicture}
&

\begin{tikzpicture}[scale = \sc, font=\scriptsize\sffamily, thick,main node/.style={circle,inner sep=0.4mm,draw, fill}]
\def\xlb{0}; \def\xub{10}; \def\ylb{0}; \def\yub{5}; \def\buf{0}; \grid
\node[main node] at (0,0) (0) {};
\node[main node] at (1,1) (1) {};
\node[main node] at (2,2) (2) {};
\node[main node] at (3,3) (3) {};
\node[main node] at (4,2) (4) {};
\node[main node] at (5,3) (5) {};
\node[main node] at (6,2) (6) {};
\node[main node] at (7,1) (7) {};
\node[main node] at (8,0) (8) {};
\node[main node] at (9,1) (9) {};
\node[main node] at (10,0) (10) {};\draw[ultra thick] (0) -- (1) -- (2) -- (3)-- (4)-- (5)-- (6)-- (7)-- (8)-- (9)-- (10);
\end{tikzpicture}
\\
$Q^{(1)}$ &
$Q^{(2)}$ &
$Q^{(3)}$ &
$Q^{(4)}$ 
\\
\\
\begin{tikzpicture}[scale = \sc, font=\scriptsize\sffamily, thick,main node/.style={circle,inner sep=0.4mm,draw, fill}]
\def\xlb{0}; \def\xub{10}; \def\ylb{0}; \def\yub{5}; \def\buf{0}; \grid
\node[main node] at (0,0) (0) {};
\node[main node] at (1,1) (1) {};
\node[main node] at (2,2) (2) {};
\node[main node] at (3,3) (3) {};
\node[main node] at (4,2) (4) {};
\node[main node] at (5,1) (5) {};
\node[main node] at (6,2) (6) {};
\node[main node] at (7,1) (7) {};
\node[main node] at (8,2) (8) {};
\node[main node] at (9,1) (9) {};
\node[main node] at (10,0) (10) {};\draw[ultra thick] (0) -- (1) -- (2) -- (3)-- (4)-- (5)-- (6)-- (7)-- (8)-- (9)-- (10);
\end{tikzpicture}
&

\begin{tikzpicture}[scale = \sc, font=\scriptsize\sffamily, thick,main node/.style={circle,inner sep=0.4mm,draw, fill}]
\def\xlb{0}; \def\xub{10}; \def\ylb{0}; \def\yub{5}; \def\buf{0}; \grid
\node[main node] at (0,0) (0) {};
\node[main node] at (1,1) (1) {};
\node[main node] at (2,2) (2) {};
\node[main node] at (3,3) (3) {};
\node[main node] at (4,2) (4) {};
\node[main node] at (5,1) (5) {};
\node[main node] at (6,2) (6) {};
\node[main node] at (7,1) (7) {};
\node[main node] at (8,0) (8) {};
\node[main node] at (9,1) (9) {};
\node[main node] at (10,0) (10) {};\draw[ultra thick] (0) -- (1) -- (2) -- (3)-- (4)-- (5)-- (6)-- (7)-- (8)-- (9)-- (10);
\end{tikzpicture}
&

\begin{tikzpicture}[scale = \sc, font=\scriptsize\sffamily, thick,main node/.style={circle,inner sep=0.4mm,draw, fill}]
\def\xlb{0}; \def\xub{10}; \def\ylb{0}; \def\yub{5}; \def\buf{0}; \grid
\node[main node] at (0,0) (0) {};
\node[main node] at (1,1) (1) {};
\node[main node] at (2,2) (2) {};
\node[main node] at (3,3) (3) {};
\node[main node] at (4,2) (4) {};
\node[main node] at (5,1) (5) {};
\node[main node] at (6,0) (6) {};
\node[main node] at (7,1) (7) {};
\node[main node] at (8,0) (8) {};
\node[main node] at (9,1) (9) {};
\node[main node] at (10,0) (10) {};\draw[ultra thick] (0) -- (1) -- (2) -- (3)-- (4)-- (5)-- (6)-- (7)-- (8)-- (9)-- (10);
\end{tikzpicture}
&

\begin{tikzpicture}[scale = \sc, font=\scriptsize\sffamily, thick,main node/.style={circle,inner sep=0.4mm,draw, fill}]
\def\xlb{0}; \def\xub{10}; \def\ylb{0}; \def\yub{5}; \def\buf{0}; \grid
\node[main node] at (0,0) (0) {};
\node[main node] at (1,1) (1) {};
\node[main node] at (2,0) (2) {};
\node[main node] at (3,1) (3) {};
\node[main node] at (4,2) (4) {};
\node[main node] at (5,3) (5) {};
\node[main node] at (6,2) (6) {};
\node[main node] at (7,3) (7) {};
\node[main node] at (8,2) (8) {};
\node[main node] at (9,1) (9) {};
\node[main node] at (10,0) (10) {};
\draw[ultra thick] (0) -- (1) -- (2) -- (3)-- (4)-- (5)-- (6)-- (7)-- (8)-- (9)-- (10);
\end{tikzpicture}
\\
$Q^{(5)} $&
$Q^{(6)}$ &
$Q^{(7)}$ &
$Q^{(8)}$ 
\\
\\
\begin{tikzpicture}[scale = \sc, font=\scriptsize\sffamily, thick,main node/.style={circle,inner sep=0.4mm,draw, fill}]
\def\xlb{0}; \def\xub{10}; \def\ylb{0}; \def\yub{5}; \def\buf{0}; \grid
\node[main node] at (0,0) (0) {};
\node[main node] at (1,1) (1) {};
\node[main node] at (2,0) (2) {};
\node[main node] at (3,1) (3) {};
\node[main node] at (4,2) (4) {};
\node[main node] at (5,3) (5) {};
\node[main node] at (6,2) (6) {};
\node[main node] at (7,1) (7) {};
\node[main node] at (8,2) (8) {};
\node[main node] at (9,1) (9) {};
\node[main node] at (10,0) (10) {};\draw[ultra thick] (0) -- (1) -- (2) -- (3)-- (4)-- (5)-- (6)-- (7)-- (8)-- (9)-- (10);
\end{tikzpicture}
&

\begin{tikzpicture}[scale = \sc, font=\scriptsize\sffamily, thick,main node/.style={circle,inner sep=0.4mm,draw, fill}]
\def\xlb{0}; \def\xub{10}; \def\ylb{0}; \def\yub{5}; \def\buf{0}; \grid
\node[main node] at (0,0) (0) {};
\node[main node] at (1,1) (1) {};
\node[main node] at (2,0) (2) {};
\node[main node] at (3,1) (3) {};
\node[main node] at (4,2) (4) {};
\node[main node] at (5,3) (5) {};
\node[main node] at (6,2) (6) {};
\node[main node] at (7,1) (7) {};
\node[main node] at (8,0) (8) {};
\node[main node] at (9,1) (9) {};
\node[main node] at (10,0) (10) {};\draw[ultra thick] (0) -- (1) -- (2) -- (3)-- (4)-- (5)-- (6)-- (7)-- (8)-- (9)-- (10);
\end{tikzpicture}
&

\begin{tikzpicture}[scale = \sc, font=\scriptsize\sffamily, thick,main node/.style={circle,inner sep=0.4mm,draw, fill}]
\def\xlb{0}; \def\xub{10}; \def\ylb{0}; \def\yub{5}; \def\buf{0}; \grid
\node[main node] at (0,0) (0) {};
\node[main node] at (1,1) (1) {};
\node[main node] at (2,0) (2) {};
\node[main node] at (3,1) (3) {};
\node[main node] at (4,0) (4) {};
\node[main node] at (5,1) (5) {};
\node[main node] at (6,2) (6) {};
\node[main node] at (7,3) (7) {};
\node[main node] at (8,2) (8) {};
\node[main node] at (9,1) (9) {};
\node[main node] at (10,0) (10) {};\draw[ultra thick] (0) -- (1) -- (2) -- (3)-- (4)-- (5)-- (6)-- (7)-- (8)-- (9)-- (10);
\end{tikzpicture}
&

\begin{tikzpicture}[scale = \sc, font=\scriptsize\sffamily, thick,main node/.style={circle,inner sep=0.4mm,draw, fill}]
\def\xlb{0}; \def\xub{10}; \def\ylb{0}; \def\yub{5}; \def\buf{0}; \grid
\node[main node] at (0,0) (0) {};
\node[main node] at (1,1) (1) {};
\node[main node] at (2,0) (2) {};
\node[main node] at (3,1) (3) {};
\node[main node] at (4,0) (4) {};
\node[main node] at (5,1) (5) {};
\node[main node] at (6,0) (6) {};
\node[main node] at (7,1) (7) {};
\node[main node] at (8,0) (8) {};
\node[main node] at (9,1) (9) {};
\node[main node] at (10,0) (10) {};\draw[ultra thick] (0) -- (1) -- (2) -- (3)-- (4)-- (5)-- (6)-- (7)-- (8)-- (9)-- (10);
\end{tikzpicture}
\\
$Q^{(9)} $&
$Q^{(10)}$ &
$Q^{(11)}$ &
$Q^{(12)}$ 
\end{tabular}

\caption{The $12$ Dyck paths in $\Q_3$ with semilength $5$}
\label{fig305}
\end{figure}
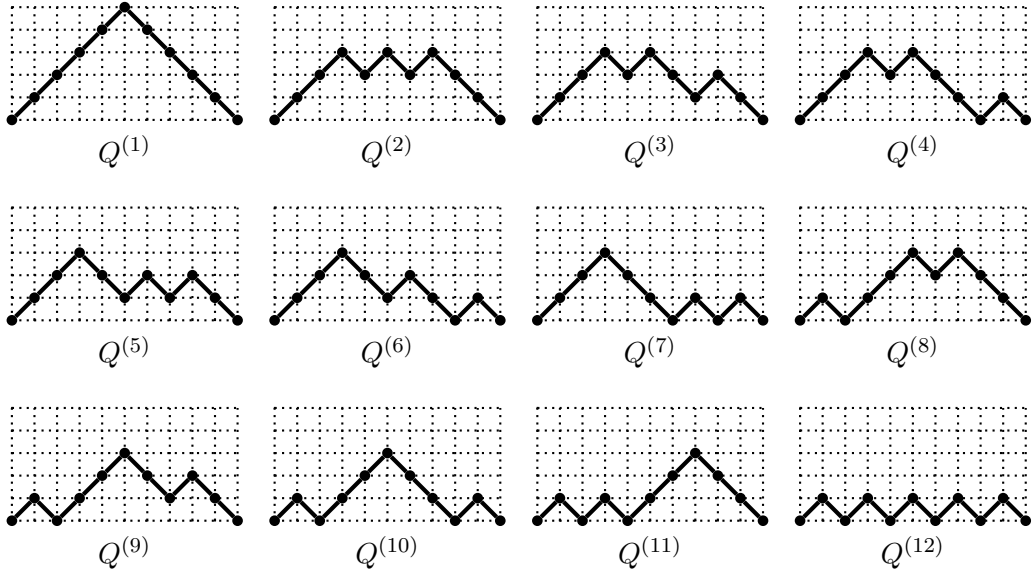

Next, we describe the bijection between rooted ordered trees and Dyck paths. Given a tree $T$, let $v_1, \ldots, v_n$ be the list of vertices in $T$ in preorder traversal (also known as depth-first traversal). Then let $j_i$ be the outdegree of $v_i$ for every $i \in [n]$. Observe that $v_n$ is always a leaf, and thus $j_n$ must be zero. We define
\[
f_3(T) \ce U^{j_1} D U^{j_2} D \cdots U^{j_{n-1}}D.
\]
For example, for the trees from Figure~\ref{fig303} and Dyck paths from Figure~\ref{fig305}, we have $f_3(T^{(i)}) = Q^{(i)}$ for every $i \in [12]$. 
Then we have the following result.

\begin{theorem}\label{thm303}
Let $d \geq 2$ and $n,k \geq 0$ be integers. Then $N_d(n,k)$ counts the number of Dyck paths in $\Q_d$ with semilength $(d-1)(n-k) + k+1$ and with exactly $k+1$ peaks.
\end{theorem}

\begin{proof}
Given $T \in \T_d$, observe that every ascent in $f_3(T)$ uniquely corresponds to an internal node in $T$. Since every nonzero outdegree in $T$ is congruent to $1$ modulo $(d-1)$, so is the length of every ascent in $f_3(T)$, and hence $f_3(T) \in \Q_d$. We also obtain that $|f_3(T)|$ is one less than the number of vertices in $T$, which is equal to the number of edges in $T$. Moreover, since the number of peaks in a Dyck path is equal to its number of ascents, the number of peaks in $f_3(T)$ is equal to
the number of internal nodes in $T$.

Conversely, given a Dyck path $Q \in \Q_d$, we can extract $j_1 \ldots, j_{n-1}$ and recover the tree $f_3^{-1}(Q) \in \T_d$. Hence, $f_3 : \T_d \to \Q_d$ is a bijection, which completes the proof.
\end{proof}

Observe that $\Q_2$ is the set of all nonempty Dyck paths. Thus, when $d=2$, Theorem~\ref{thm303} specializes to the well-known fact that $N_2(n,k)$ counts the number of Dyck paths with semilength $n+1$ and with exactly $k+1$ peaks~\cite[Section 2.4.2]{Petersen15}.


\subsection{$231$-avoiding permutations with restricted decreasing runs}

Let $P \colon [n] \to [n]$ be a permutation of $[n]$. For convenience, we will often write $P = P_1 \cdots P_n$ to indicate that the permutation $P$ maps $i$ to $P_i$ for every $i \in [n]$. Also, let $\pi = \pi_1 \cdots \pi_k$ be a permutation of $[k]$ where $k \leq n$. We say that $P$ \emph{contains} the pattern $\pi$ if there exist indices $1 \leq i_1 < i_2 < \cdots < i_k$ such that, for every distinct $j, j' \in [k]$, $P_{i_j} < P_{i_{j'}}$ if and only if $\pi_j < \pi_{j'}$. For example, $P = 43521$ contains (four instances of) the pattern $\pi = 231$. Conversely, we say that $P$ \emph{avoids} $\pi$ if $P$ does not contain $\pi$. Also, a \emph{decreasing run} in a permutation $P$ is a maximal decreasing subsequence of $P$. For example, $P = 314652$ contains three decreasing runs: $31$, $4$, and $652$.

Given an integer $d \geq 2$, we let $\P_d$ denote the set of nonempty $231$-avoiding permutations where every decreasing run has length congruent to $1$ modulo $(d-1)$. For example, Table~\ref{tab301} lists the $12$ permutations of $[5]$ in $\P_3$.

\begin{table}
\[
\begin{array}{llll}
P^{(1)} \ce 54321 &\quad P^{(2)} \ce 54123 &\quad P^{(3)} \ce 53124 &\quad P^{(4)} \ce 43125\\
P^{(5)} \ce 52134 &\quad P^{(6)} \ce 42135 &\quad P^{(7)} \ce 32145 &\quad P^{(8)} \ce 15423\\
P^{(9)} \ce 15324 &\quad P^{(10)} \ce 14325 &\quad P^{(11)} \ce 12543 &\quad P^{(12)} \ce 12345
\end{array}
\]
\caption{The $12$ permutations of $[5]$ in $\P_3$}\label{tab301}
\end{table}

Recall the notion of the matching down step of a particular up step in a Schr\"{o}der path, which also applies to Dyck paths. Now consider the function $f_4 : \Q_d \to \P_d$ defined as follows. Given a Dyck path $Q$ of semilength $n$, define the permutation $P = f_4(Q)$ on $[n]$ such that $P_i = j$ if and only if the $j$-th down step in $Q$ matches with the $i$-th up step in $Q$. For example, consider the Dyck path
\[
Q = U_1 U_2 U_3 D_1 D_2 U_4 D_3 D_4 U_5 U_6 D_5 D_6,
\]
where each up and down step is labelled by the order they appear in $Q$. Notice that $D_1$ is the matching down step of $U_3$, and thus the image of $Q$ under $f_4$ would map $3$ to $1$. In fact, one can check that $f_4(Q) = 421365$. Also, for the Dyck paths in Figure~\ref{fig305} and the permutations in Table~\ref{tab301}, we have $f_4( Q^{(i)}) = P^{(i)}$ for every $i \in [12]$. Then we have the following.

\begin{theorem}\label{thm304}
Let $d \geq 2$ and $n,k \geq 0$ be integers. Then $N_d(n,k)$ counts the number of permutations of $[(d-1)(n-k) + k+1]$ in $\P_d$ with exactly $k+1$ decreasing runs.
\end{theorem}

\begin{proof}
It is known that $f_4$ gives a bijection between Dyck paths and $231$-avoiding permutations (see, for instance,~\cite[Section 2.4.3]{Petersen15} for the details). Moreover, notice that there is a one-to-one correspondence between ascents in a Dyck path $Q$ and decreasing runs in $f_4(Q)$. Thus, given $Q \in \Q_d$, every decreasing run in $f_4(Q)$ has length congruent to $1$ modulo $(d-1)$, and thus $f_4(Q) \in \P_d$. Furthermore, we see that $f_4(Q)$ is a permutation of $[|Q|]$ and has as many decreasing
runs as $Q$ has ascents. Thus, our claim follows.
\end{proof}

Observe that $\P_2$ is the set of all nonempty $231$-avoiding permutations. Thus, Theorem~\ref{thm304} extends the well-known fact that $N_2(n,k)$ gives the number of $231$-avoiding permutations of $[n+1]$ with $k+1$ decreasing runs~\cite[Section 2.4.3]{Petersen15}. Furthermore, one can use known bijections between $231$-avoiding permutations and other classic combinatorial objects (such as non-crossing partitions, full binary trees, and standard Young tableaux --- see, for instance,~\cite[Chapter 2]{Petersen15}) to obtain yet more combinatorial interpretations of $N_d(n,k)$.

\section{Several ``replicative'' combinatorial interpretations of $N_d(n,k)$}\label{sec4}

In this section, we describe four more combinatorial interpretations of $N_d(n,k)$. All four sets of combinatorial objects can be seen as generalizations of objects shown in Huh et al.~\cite{HuhKSS24} to be counted by $N_3(n,k)$.

\subsection{Schr\"{o}der paths with labelled descents}

Given an integer $d \geq 2$, let $\widetilde{\S}_d$ denote the set of Schr\"{o}der paths in which:
\begin{itemize}
    \item every descent has length at most $d-1$;
    \item every descent of length $\ell \geq 1$ is labelled by an $(\ell - 1)$-element subset of $[d-2]$.
\end{itemize}

For example, an element in $\widetilde{\S}_4$ is
\[
UUUHH\underset{\emptyset}{\underline{D}}U\underset{\set{1}}{\underline{DD}}HUU\underset{\set{1,2}}{\underline{DDD}}H\underset{\emptyset}{\underline{D}}.
\]
Observe that descents of length $1$ are always labelled by the empty set, and when $d=3$, descents of length $2$ are always labelled by the set $\set{1}$. While it can be practically convenient to consider these descents unlabelled, we will preserve these labels to help streamline our description of a bijection from $\M_d$ to $\widetilde{\S}_d$. Given $M \in \M_d$, we define $f_5(M) \in \widetilde{\S}_d$ recursively as follows:

\begin{itemize}
    \item If $M$ is a single indeterminate, then $f_5(M)$ is the empty path.
    \item If $M = L(M')$ for some monomial $M' \in \M_d$, then $f_5(M) \ce f_5(M')H$.
    \item Otherwise, we can write $M = M_0 M_1 M_2 \cdots M_{d-1}$ such that $M_0 \in \M_d$, and $M_i \in \overline{\M}_d$ for every $i \in [d-1]$. Let $I \ce \set{ i \in [d-2] : \lopt(M_i) \geq 1}$, and define indices $1 \leq i_1 < \cdots < i_{|I|} \leq d-2$ such that $I = \set{i_1, \ldots, i_{|I|}}$. Also, for every $i \in I$, define $M_i'$ such that $M_i = L(M_i')$. Then we define
    \[
    f_5(M) \ce f_5(M_0) U f_5( M_{i_1}') U f_5( M_{i_2}') \cdots U f_5( M_{i_{|I|}}') U  f_5(M_{d-1})
\underset{I}{\underline{D^{|I|+1}}}.
    \]
\end{itemize}

For example, given $M \ce L(a_1 L(a_2)L^2(a_3a_4a_5)) \in \M_3$, we have
\begin{align*}
f_5(M)
&=  f_5( a_1 L(a_2)L^2(a_3a_4a_5) H\\
&= f_5(a_1) U f_5(a_2) U f_5( L^2(a_3a_4a_5)) \underset{\set{1}}{\underline{DD}} H. 
\end{align*}
Now $f_5(a_1)$ and $f_5(a_2)$ are both the empty path, while 
\[
f_5( L^2(a_3a_4a_5)) = f_5(a_3a_4a_5)HH = U\underset{\emptyset}{\underline{D}}HH.
\]
Thus, we obtain that $f_5(M) = UUU\underset{\emptyset}{\underline{D}}HH \underset{\set{1}}{\underline{DD}}H$. Observe that $f_5(M)$ has semilength $6$, which is equal to $\topt(M)$. Also, $f_5(M)$ has $3$ instances of $H$ and $1$ instance of $DD$, which sum to $\lopt(M) = 4$. Figure~\ref{fig401} lists the Schr\"{o}der paths $f_5(M)$ for all $21$ monomials $M \in \M_3$ with $\topt(M)=3$. Again, since every descent of length $1$ and $2$ are labelled by $\emptyset$ and $\set{1}$ respectively in this case, we have suppressed these labels in Figure~\ref{fig401} to reduce cluttering.

\begin{figure}[htbp]
\centering
\scriptsize
\def\sc{0.25}
\def\grid{\foreach \x in {\xlb ,...,\xub}
{    \ifthenelse{\NOT 0 = \x}{\draw[thick](\x ,-1pt) -- (\x ,1pt);}{}
\draw[dotted, thick](\x,\ylb- \buf) -- (\x,\yub + \buf);}
\foreach \y in {\ylb ,...,\yub}
{    \ifthenelse{\NOT 0 = \y}{\draw[thick](-1pt, \y) -- (1pt, \y);}{}
\draw[dotted, thick](\xlb- \buf, \y) -- (\xub + \buf, \y);}
}

\begin{tabular}{cccccc}
\begin{tikzpicture}[scale = \sc, font=\scriptsize\sffamily, thick,main node/.style={circle,inner sep=0.4mm,draw, fill}]
\def\xlb{0}; \def\xub{6}; \def\ylb{0}; \def\yub{2}; \def\buf{0}; \grid
\node[main node] at (0,0) (0) {};
\node[main node] at (1,1) (1) {};
\node[main node] at (2,0) (2) {};
\node[main node] at (3,1) (3) {};
\node[main node] at (4,0) (4) {};
\node[main node] at (5,1) (5) {};
\node[main node] at (6,0) (6) {};
\draw[ultra thick] (0) -- (1) -- (2) -- (3)-- (4)-- (5)-- (6);
\end{tikzpicture}
&

\begin{tikzpicture}[scale = \sc, font=\scriptsize\sffamily, thick,main node/.style={circle,inner sep=0.4mm,draw, fill}]
\def\xlb{0}; \def\xub{6}; \def\ylb{0}; \def\yub{2}; \def\buf{0}; \grid
\node[main node] at (0,0) (0) {};
\node[main node] at (1,1) (1) {};
\node[main node] at (2,0) (2) {};
\node[main node] at (3,1) (3) {};
\node[main node] at (4,0) (4) {};
\node[main node] at (6,0) (5) {};
\draw[ultra thick] (0) -- (1) -- (2) -- (3)-- (4)-- (5);

\end{tikzpicture}
&
\begin{tikzpicture}[scale = \sc, font=\scriptsize\sffamily, thick,main node/.style={circle,inner sep=0.4mm,draw, fill}]
\def\xlb{0}; \def\xub{6}; \def\ylb{0}; \def\yub{2}; \def\buf{0}; \grid
\node[main node] at (0,0) (0) {};
\node[main node] at (1,1) (1) {};
\node[main node] at (2,0) (2) {};
\node[main node] at (4,0) (4) {};
\node[main node] at (5,1) (5) {};
\node[main node] at (6,0) (6) {};
\draw[ultra thick] (0) -- (1) -- (2) -- (4)-- (5)-- (6);
\end{tikzpicture}
&
\begin{tikzpicture}[scale = \sc, font=\scriptsize\sffamily, thick,main node/.style={circle,inner sep=0.4mm,draw, fill}]
\def\xlb{0}; \def\xub{6}; \def\ylb{0}; \def\yub{2}; \def\buf{0}; \grid
\node[main node] at (0,0) (0) {};
\node[main node] at (2,0) (2) {};
\node[main node] at (3,1) (3) {};
\node[main node] at (4,0) (4) {};
\node[main node] at (5,1) (5) {};
\node[main node] at (6,0) (6) {};
\draw[ultra thick] (0) -- (2) -- (3)-- (4)-- (5)-- (6);\end{tikzpicture}
&
\begin{tikzpicture}[scale = \sc, font=\scriptsize\sffamily, thick,main node/.style={circle,inner sep=0.4mm,draw, fill}]
\def\xlb{0}; \def\xub{6}; \def\ylb{0}; \def\yub{2}; \def\buf{0}; \grid
\node[main node] at (0,0) (0) {};
\node[main node] at (1,1) (1) {};
\node[main node] at (2,0) (2) {};
\node[main node] at (3,1) (3) {};
\node[main node] at (4,2) (4) {};
\node[main node] at (5,1) (5) {};
\node[main node] at (6,0) (6) {};
\draw[ultra thick] (0) -- (1) -- (2) -- (3)-- (4)-- (5)-- (6);
\end{tikzpicture}
&
\begin{tikzpicture}[scale = \sc, font=\scriptsize\sffamily, thick,main node/.style={circle,inner sep=0.4mm,draw, fill}]
\def\xlb{0}; \def\xub{6}; \def\ylb{0}; \def\yub{2}; \def\buf{0}; \grid
\node[main node] at (0,0) (0) {};
\node[main node] at (1,1) (1) {};
\node[main node] at (2,2) (2) {};
\node[main node] at (3,1) (3) {};
\node[main node] at (4,0) (4) {};
\node[main node] at (5,1) (5) {};
\node[main node] at (6,0) (6) {};
\draw[ultra thick] (0) -- (1) -- (2) -- (3)-- (4)-- (5)-- (6);
\end{tikzpicture}
\\

$a_1a_2a_3a_4a_5a_6a_7$ &
$L(a_1a_2a_3a_4a_5)$ &
$L(a_1a_2a_3)a_4a_5$&
$L(a_1)a_2a_3a_4a_5$&
$a_1L(a_2a_3a_4)a_5$&
$ a_1L(a_2)a_3a_4a_5$\\
\\

\begin{tikzpicture}[scale = \sc, font=\scriptsize\sffamily, thick,main node/.style={circle,inner sep=0.4mm,draw, fill}]
\def\xlb{0}; \def\xub{6}; \def\ylb{0}; \def\yub{2}; \def\buf{0}; \grid
\node[main node] at (0,0) (0) {};
\node[main node] at (1,1) (1) {};
\node[main node] at (2,2) (2) {};
\node[main node] at (3,1) (3) {};
\node[main node] at (5,1) (5) {};
\node[main node] at (6,0) (6) {};
\draw[ultra thick] (0) -- (1) -- (2) -- (3)-- (5)-- (6);\end{tikzpicture}
&
\begin{tikzpicture}[scale = \sc, font=\scriptsize\sffamily, thick,main node/.style={circle,inner sep=0.4mm,draw, fill}]
\def\xlb{0}; \def\xub{6}; \def\ylb{0}; \def\yub{2}; \def\buf{0}; \grid
\node[main node] at (0,0) (0) {};
\node[main node] at (1,1) (1) {};
\node[main node] at (3,1) (3) {};
\node[main node] at (4,0) (4) {};
\node[main node] at (5,1) (5) {};
\node[main node] at (6,0) (6) {};
\draw[ultra thick] (0) -- (1)-- (3)-- (4)-- (5)-- (6);\end{tikzpicture}
&
\begin{tikzpicture}[scale = \sc, font=\scriptsize\sffamily, thick,main node/.style={circle,inner sep=0.4mm,draw, fill}]
\def\xlb{0}; \def\xub{6}; \def\ylb{0}; \def\yub{2}; \def\buf{0}; \grid
\node[main node] at (0,0) (0) {};
\node[main node] at (1,1) (1) {};
\node[main node] at (2,0) (2) {};
\node[main node] at (3,1) (3) {};
\node[main node] at (4,2) (4) {};
\node[main node] at (5,1) (5) {};
\node[main node] at (6,0) (6) {};
\draw[ultra thick] (0) -- (1) -- (2) -- (3)-- (4)-- (5)-- (6);\end{tikzpicture}
&

\begin{tikzpicture}[scale = \sc, font=\scriptsize\sffamily, thick,main node/.style={circle,inner sep=0.4mm,draw, fill}]
\def\xlb{0}; \def\xub{6}; \def\ylb{0}; \def\yub{2}; \def\buf{0}; \grid
\node[main node] at (0,0) (0) {};
\node[main node] at (1,1) (1) {};
\node[main node] at (2,0) (2) {};
\node[main node] at (3,1) (3) {};
\node[main node] at (5,1) (5) {};
\node[main node] at (6,0) (6) {};
\draw[ultra thick] (0) -- (1) -- (2) -- (3)--  (5)-- (6);
\end{tikzpicture}
&
\begin{tikzpicture}[scale = \sc, font=\scriptsize\sffamily, thick,main node/.style={circle,inner sep=0.4mm,draw, fill}]
\def\xlb{0}; \def\xub{6}; \def\ylb{0}; \def\yub{2}; \def\buf{0}; \grid
\node[main node] at (0,0) (0) {};
\node[main node] at (1,1) (1) {};
\node[main node] at (2,0) (2) {};
\node[main node] at (4,0) (4) {};
\node[main node] at (6,0) (6) {};
\draw[ultra thick] (0) -- (1) -- (2) --  (4)--  (6);
\end{tikzpicture}
&
\begin{tikzpicture}[scale = \sc, font=\scriptsize\sffamily, thick,main node/.style={circle,inner sep=0.4mm,draw, fill}]
\def\xlb{0}; \def\xub{6}; \def\ylb{0}; \def\yub{2}; \def\buf{0}; \grid
\node[main node] at (0,0) (0) {};
\node[main node] at (2,0) (2) {};
\node[main node] at (3,1) (3) {};
\node[main node] at (4,0) (4) {};
\node[main node] at (6,0) (6) {};
\draw[ultra thick] (0) -- (2) -- (3)-- (4)-- (6);
\end{tikzpicture}
\\

$  a_1a_2L(a_3a_4a_5)$&
$  a_1a_2L(a_3)a_4a_5$&
$a_1a_2a_3L(a_4)a_5$&
$a_1a_2a_3a_4L(a_5)$&
$ L(L(a_1a_2a_3))$&
$ L(L(a_1)a_2a_3)$ \\
\\

\begin{tikzpicture}[scale = \sc, font=\scriptsize\sffamily, thick,main node/.style={circle,inner sep=0.4mm,draw, fill}]
\def\xlb{0}; \def\xub{6}; \def\ylb{0}; \def\yub{2}; \def\buf{0}; \grid
\node[main node] at (0,0) (0) {};
\node[main node] at (2,0) (2) {};
\node[main node] at (4,0) (4) {};
\node[main node] at (5,1) (5) {};
\node[main node] at (6,0) (6) {};
\draw[ultra thick] (0) -- (2) -- (4)-- (5)-- (6);
\end{tikzpicture}
&

\begin{tikzpicture}[scale = \sc, font=\scriptsize\sffamily, thick,main node/.style={circle,inner sep=0.4mm,draw, fill}]
\def\xlb{0}; \def\xub{6}; \def\ylb{0}; \def\yub{2}; \def\buf{0}; \grid
\node[main node] at (0,0) (0) {};
\node[main node] at (1,1) (1) {};
\node[main node] at (2,2) (2) {};
\node[main node] at (3,1) (3) {};
\node[main node] at (4,0) (4) {};
\node[main node] at (6,0) (6) {};
\draw[ultra thick] (0) -- (1) -- (2) -- (3)-- (4)--  (6);
\end{tikzpicture}
&

\begin{tikzpicture}[scale = \sc, font=\scriptsize\sffamily, thick,main node/.style={circle,inner sep=0.4mm,draw, fill}]
\def\xlb{0}; \def\xub{6}; \def\ylb{0}; \def\yub{2}; \def\buf{0}; \grid
\node[main node] at (0,0) (0) {};
\node[main node] at (1,1) (1) {};
\node[main node] at (3,1) (3) {};
\node[main node] at (4,0) (4) {};
\node[main node] at (6,0) (6) {};
\draw[ultra thick] (0) -- (1) -- (3)-- (4)-- (6);
\end{tikzpicture}
&
  
\begin{tikzpicture}[scale = \sc, font=\scriptsize\sffamily, thick,main node/.style={circle,inner sep=0.4mm,draw, fill}]
\def\xlb{0}; \def\xub{6}; \def\ylb{0}; \def\yub{2}; \def\buf{0}; \grid
\node[main node] at (0,0) (0) {};
\node[main node] at (2,0) (2) {};
\node[main node] at (3,1) (3) {};
\node[main node] at (4,2) (4) {};
\node[main node] at (5,1) (5) {};
\node[main node] at (6,0) (6) {};
\draw[ultra thick] (0) -- (2) -- (3)-- (4)-- (5)-- (6);
\end{tikzpicture}
&

\begin{tikzpicture}[scale = \sc, font=\scriptsize\sffamily, thick,main node/.style={circle,inner sep=0.4mm,draw, fill}]
\def\xlb{0}; \def\xub{6}; \def\ylb{0}; \def\yub{2}; \def\buf{0}; \grid
\node[main node] at (0,0) (0) {};
\node[main node] at (2,0) (2) {};
\node[main node] at (3,1) (3) {};
\node[main node] at (5,1) (5) {};
\node[main node] at (6,0) (6) {};
\draw[ultra thick] (0) -- (2) -- (3)-- (5)-- (6);
;\end{tikzpicture}
&

\begin{tikzpicture}[scale = \sc, font=\scriptsize\sffamily, thick,main node/.style={circle,inner sep=0.4mm,draw, fill}]
\def\xlb{0}; \def\xub{6}; \def\ylb{0}; \def\yub{2}; \def\buf{0}; \grid
\node[main node] at (0,0) (0) {};
\node[main node] at (1,1) (1) {};
\node[main node] at (3,1) (3) {};
\node[main node] at (4,2) (4) {};
\node[main node] at (5,1) (5) {};
\node[main node] at (6,0) (6) {};
\draw[ultra thick] (0) -- (1) -- (3)-- (4)-- (5)-- (6);
\end{tikzpicture}
\\
\\

$ L(L(a_1))a_2a_3 $&
$ L(a_1L(a_2)a_3)$&
$ L(a_1a_2L(a_3)) $&
$ L(a_1)L(a_2)a_3 $&
$L(a_1)a_2L(a_3) $&
$a_1L(L(a_2))a_3 $\\
\\

\begin{tikzpicture}[scale = \sc, font=\scriptsize\sffamily, thick,main node/.style={circle,inner sep=0.4mm,draw, fill}]
\def\xlb{0}; \def\xub{6}; \def\ylb{0}; \def\yub{2}; \def\buf{0}; \grid
\node[main node] at (0,0) (0) {};
\node[main node] at (1,1) (1) {};
\node[main node] at (2,2) (2) {};
\node[main node] at (4,2) (4) {};
\node[main node] at (5,1) (5) {};
\node[main node] at (6,0) (6) {};
\draw[ultra thick] (0) -- (1) -- (2) -- (4)-- (5)-- (6);
\end{tikzpicture}
&

\begin{tikzpicture}[scale = \sc, font=\scriptsize\sffamily, thick,main node/.style={circle,inner sep=0.4mm,draw, fill}]
\def\xlb{0}; \def\xub{6}; \def\ylb{0}; \def\yub{2}; \def\buf{0}; \grid
\node[main node] at (0,0) (0) {};
\node[main node] at (1,1) (1) {};
\node[main node] at (3,1) (3) {};
\node[main node] at (5,1) (5) {};
\node[main node] at (6,0) (6) {};
\draw[ultra thick] (0) -- (1)  -- (3)-- (5)-- (6);
\end{tikzpicture}
&

\begin{tikzpicture}[scale = \sc, font=\scriptsize\sffamily, thick,main node/.style={circle,inner sep=0.4mm,draw, fill}]
\def\xlb{0}; \def\xub{6}; \def\ylb{0}; \def\yub{2}; \def\buf{0}; \grid
\node[main node] at (0,0) (0) {};
\node[main node] at (2,0) (2) {};
\node[main node] at (4,0) (4) {};
\node[main node] at (6,0) (6) {};
\draw[ultra thick] (0) -- (2) --  (4)--  (6);
\end{tikzpicture}
\\
$a_1L(a_2)L(a_3) $&
$ a_1a_2L(L(a_3)) $&
$L(L(L(a_1)))$
\end{tabular}
\caption{Illustrating the bijection $f_5$ for all $M \in \M_3$ with $\topt(M)=3$ (with descent labels suppressed)}\label{fig401}
\end{figure}
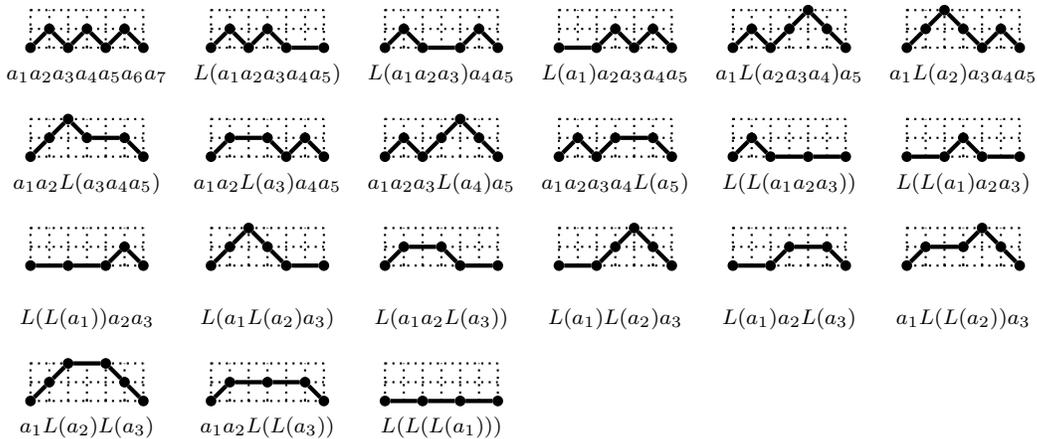

Theorem~\ref{thm401} shows that $\widetilde{\S}_d$ provides yet another combinatorial interpretation for $N_d(n,k)$.

\begin{theorem}\label{thm401}
Let $d \geq 2$ and $n,k \geq 0$ be integers. Then $N_d(n,k)$ counts the number of Schr\"{o}der paths in $\widetilde{\S}_d$  with semilength $n$ and with exactly $k$ occurrences in total of $H$ and $DD$.
\end{theorem}

\begin{proof}
For convenience, given a Schr\"{o}der path $S$, we let $\hdd(S)$ denote the total number of instances of $H$ and $DD$ in $S$. Let $M \in \M_d$. We examine the mapping $f_5$ case by case.

If $M$ is the single indeterminate, then $\topt(M) = \lopt(M) = 0$, and $f_5(M)$, the empty path, also has $|f_5(M)| = \hdd(f_5(M)) = 0$.

If $M = L(M')$ for some $M' \in \M_d$, then $\topt(M') = \topt(M) -1$ and $\lopt(M') = \lopt(M) -1$. Now $f_5(M) = f_5(M')H$, and so we also have $| f_5(M') | = |f_5(M)| -1$ and $\hdd(f_5(M')) = \hdd(f_5(M)) -1$.

Now suppose $M = M_0 M_1\cdots M_{d-1}$ where $M_i$ is irreducible for every $i \in [d-1]$. From the definition of $f_5$, we have
\[
f_5(M) = f_5(M_0) U f_5( M_{i_1}') U f_5( M_{i_2}') \cdots U f_5( M_{i_{|I|}}) U  f_5(M_{d-1}) \underset{I}{
\underline{D^{|I|+1}}}.
\]
In this case, we have 
\begin{align*}
\topt(f_5(M)) &= \topt(M_0) + \sum_{i \in I} \topt( M_i') + \topt(M_{d-1}) + 1 + |I|,\\
\lopt(f_5(M)) &= \lopt(M_0) + \sum_{i \in I} \lopt( M_i') + \lopt(M_{d-1}) + |I|.
\end{align*}
We also have $|f_5(M)| = | f_5(M_0)| + \sum_{i \in I} | f_5(M_i')| + | f_5(M_{d-1})| + 1 + |I|$. Notice that given any irreducible monomial $M$, $f_5(M)$ is either the empty path or must end with an $H$, and thus $f_5(M_{d-1})$ cannot end with a $D$. Therefore, $f_5(M)$ ends with a descent of length $|I|+1$, which contains $|I|$ instances of $DD$. Hence, $\hdd(f_5(M)) = \hdd( f_5(M_0)) + \sum_{i \in I} \hdd( f_5(  M_i')) + \hdd(f_5(M_{d-1})) + |I|$.

Since $f_5$ never produces a descent of length at least $d$ and labels every descent of length $\ell \geq 1$ with an $(\ell-1)$-element subset of $[d-2]$, we have $f_5(M) \in \widetilde{\S}_d$ for every $M \in \M_d$, with $|f_5(M)| = \topt(M)$ and $\hdd(f_5(M)) = \lopt(M)$. To complete our proof, it suffices to show that the inverse function of $f_5$ exists.

Let $S \in \widetilde{\S}_{d}$. If $S$ is the empty path, then $f_5^{-1}(S)$ is the monomial with a single indeterminate. Otherwise, $S$ must either end with an $H$ or a $D$. If $S = S'H$ for some $S' \in \widetilde{\S}_d$, then $f_5^{-1}(S) = L (f_5^{-1}(S'))$. Otherwise, suppose $S$ ends with a descent of length $\ell \geq 1$. Then we can uniquely write
\[
S = S^{(0)} U S^{(1)} U S^{(2)} \cdots U S^{(\ell)} \underset{I}{ \underline{D^{\ell}}}
\]
for some $S^{(0)}, \ldots, S^{(\ell)} \in \widetilde{\S}_d$, with an $(\ell-1)$-element set $I \subseteq [d-2]$ marking the terminal descent of $S$. Then we obtain that $f_5^{-1}(S) = M_0M_1 \cdots M_d$ where
\begin{itemize}
\item
$M_0 = f_5^{-1} (S^{(0)})$ and $M_{d-1} = f_5^{-1} ( S^{(\ell)})$.
\item
For every $i \in [d-2]$, if $i \not\in I$, then $M_i$ is a single indeterminate. Otherwise, if $i$ is the $j$-th smallest index in $I$, then $M_i = L( f_5^{-1}( S^{(j)}))$.
\end{itemize}
Thus, the mapping $f_5$ is indeed invertible, and our claim follows.
\end{proof}

Again, given $S \in \widetilde{\S}_d$, every descent of length $1$ is labelled by $\emptyset$. Thus, we can ignore these markings, in which case $\widetilde{\S}_2$ gives exactly the (unlabelled) Schr\"{o}der paths without descents of length at least $2$, which is known to be counted by the Catalan numbers~\cite[Chapter 2, Exercise 49]{Stanley15}. Likewise, when $d=3$, every instance of $DD$ is labelled
by the set $\set{1}$, and so we can ignore these markings as well. In that case, our theorem specializes to Huh et al.'s result~\cite[Theorem 6.2(2) and Corollary 6.3]{HuhKSS24}, which shows that $N_3(n,k)$ counts the number of Schr\"{o}der paths $S$ where $|S|=n$, $\hdd(S)=k$, and $S$ does not contain descents of length at least $3$.

\subsection{Generalized $F$-paths}

We next describe a set of lattice paths that can be seen as a generalization of the notion of $F$-paths defined in Huh et al.~\cite{HuhKSS24}. Given an integer $d \geq 2$, let $\F_d$ denote the set of lattice paths such that:

\begin{itemize}
\item the path begins at $(0,0)$, and always stays on or above the line $y=x$;
\item every step has the form $(\ell,1)$ for some integer $\ell \geq 0$;
\item every step of the form $(\ell,1)$ with $\ell \geq 1$ is labelled by a composition of $\ell-1$ consisting of $d-1$ nonnegative parts (i.e., an ordered list of $d-1$ nonnegative integers which sum to $\ell - 1$).
\end{itemize}

For example,
\[
F=  (0,1),  (0,1), \underset{(0,0)}{\underline{(1,1)}}, (0,1), \underset{(0,2)}{\underline{(3,1)}}, \underset{(1,0)}{\underline{(2,1)}}
\]
is an element of $\F_3$. Also, given $F \in \F_d$, we let $|F|$ denote the number of steps in $F$. For instance, Figure~\ref{fig402} illustrates the $6$ elements in $\F_3$ where $|F| = 2$. 

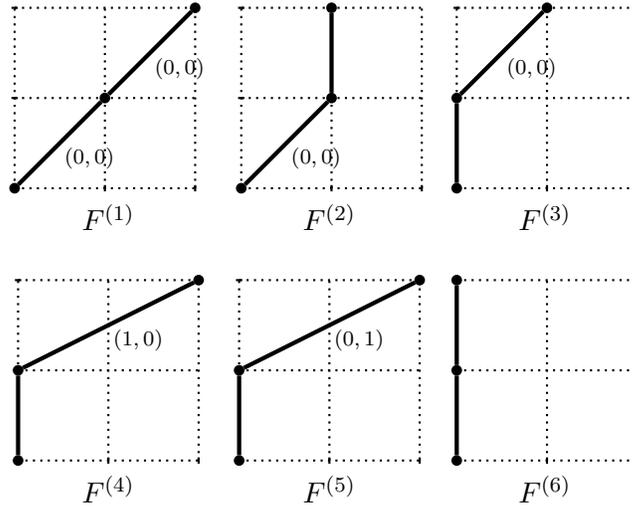
\begin{figure}[htbp]
\centering
\def\sc{1.2}
\def\grid{\foreach \x in {\xlb ,...,\xub}
{    \ifthenelse{\NOT 0 = \x}{\draw[thick](\x ,-1pt) -- (\x ,1pt);}{}
\draw[dotted, thick](\x,\ylb- \buf) -- (\x,\yub + \buf);}
\foreach \y in {\ylb ,...,\yub}
{    \ifthenelse{\NOT 0 = \y}{\draw[thick](-1pt, \y) -- (1pt, \y);}{}
\draw[dotted, thick](\xlb- \buf, \y) -- (\xub + \buf, \y);}
}

\begin{tabular}{ccc}
\begin{tikzpicture}[scale = \sc, font=\scriptsize\sffamily, thick,main node/.style={circle,inner sep=0.4mm,draw, fill}]
\def\xlb{0}; \def\xub{2}; \def\ylb{0}; \def\yub{2}; \def\buf{0}; \grid
\node[main node] at (0,0) (0) {};
\node[main node] at (1,1) (1) {};
\node[main node] at (2,2) (2) {};
\draw[ultra thick] (0) --node[pos=0.5, below right, inner sep=1pt right, inner sep=1pt] {$(0,0)$} (1) --node[pos=0.5, below right, inner sep=1pt right, inner sep=1pt] {$(0,0)$} (2);
\end{tikzpicture}
&
\begin{tikzpicture}[scale = \sc, font=\scriptsize\sffamily, thick,main node/.style={circle,inner sep=0.4mm,draw, fill}]
\def\xlb{0}; \def\xub{2}; \def\ylb{0}; \def\yub{2}; \def\buf{0}; \grid
\node[main node] at (0,0) (0) {};
\node[main node] at (1,1) (1) {};
\node[main node] at (1,2) (2) {};
\draw[ultra thick] (0) --node[pos=0.5, below right, inner sep=1pt] {$(0,0)$} (1) -- (2);
\end{tikzpicture}
&
\begin{tikzpicture}[scale = \sc, font=\scriptsize\sffamily, thick,main node/.style={circle,inner sep=0.4mm,draw, fill}]
\def\xlb{0}; \def\xub{2}; \def\ylb{0}; \def\yub{2}; \def\buf{0}; \grid
\node[main node] at (0,0) (0) {};
\node[main node] at (0,1) (1) {};
\node[main node] at (1,2) (2) {};
\draw[ultra thick] (0) -- (1) --node[pos=0.5, below right, inner sep=1pt] {$(0,0)$} (2);
\end{tikzpicture}
\\
$F^{(1)}$&
$F^{(2)}$&
$F^{(3)}$
\\
\\
\begin{tikzpicture}[scale = \sc, font=\scriptsize\sffamily, thick,main node/.style={circle,inner sep=0.4mm,draw, fill}]
\def\xlb{0}; \def\xub{2}; \def\ylb{0}; \def\yub{2}; \def\buf{0}; \grid
\node[main node] at (0,0) (0) {};
\node[main node] at (0,1) (1) {};
\node[main node] at (2,2) (2) {};
\draw[ultra thick] (0) -- (1) --node[pos=0.5, below right, inner sep=1pt] {$(1,0)$} (2);
\end{tikzpicture}
&
\begin{tikzpicture}[scale = \sc, font=\scriptsize\sffamily, thick,main node/.style={circle,inner sep=0.4mm,draw, fill}]
\def\xlb{0}; \def\xub{2}; \def\ylb{0}; \def\yub{2}; \def\buf{0}; \grid
\node[main node] at (0,0) (0) {};
\node[main node] at (0,1) (1) {};
\node[main node] at (2,2) (2) {};
\draw[ultra thick] (0) -- (1) --node[pos=0.5, below right, inner sep=1pt] {$(0,1)$} (2);
\end{tikzpicture}
&
\begin{tikzpicture}[scale = \sc, font=\scriptsize\sffamily, thick,main node/.style={circle,inner sep=0.4mm,draw, fill}]
\def\xlb{0}; \def\xub{2}; \def\ylb{0}; \def\yub{2}; \def\buf{0}; \grid
\node[main node] at (0,0) (0) {};
\node[main node] at (0,1) (1) {};
\node[main node] at (0,2) (2) {};
\draw[ultra thick] (0) -- (1) -- (2);
\end{tikzpicture}
\\
$F^{(4)}$ &
$F^{(5)}$ &
$F^{(6)}$ 
\end{tabular}

\caption{The $6$ lattice paths $F \in \F_3$ with $|F| = 2$}\label{fig402}
\end{figure}
Observe from the definition of \( \mathcal{F}_d \) that the step \( (1,1) \) is always labelled by the composition where every part is \( 0 \). As with the elements in \( \widetilde{\mathcal{S}}_d \), preserving these somewhat redundant markings will be helpful in establishing our results.

Next, given \( M \in \mathcal{M}_d \), we define \( \lofi(M) \) (\emph{linear operators on first indeterminate}) to be the number of linear operators \( L \) applied to the first (i.e., leftmost) indeterminate in \( M \). For instance, among \( M \in \mathcal{M}_3 \) with \( \topt(M) = 3 \), the seven elements with \( \lofi(M) = 1 \) are
\[
\begin{array}{llll}
L(a_1 a_2 a_3 a_4 a_5), &
 L(a_1 a_2 a_3) a_4 a_5, &
  L(a_1) a_2 a_3 a_4 a_5,&
   L(a_1 (L(a_2) a_3)), \\[3pt]
   L(a_1 a_2 L(a_3)),&
    L(a_1) L(a_2) a_3, &
    L(a_1) a_2 L(a_3).
\end{array}
\]
Observe that \( 0 \leq \lofi(M) \leq \lopt(M) \) for every \( M \in \mathcal{M}_d \). Next, given \( F \in \mathcal{F}_d \), let \( \north(F) \) denote the number of \( (0,1) \) steps in \( F \). Also, let \( \height(F) \) denote \( y_n - x_n \), where \( (x_n, y_n) \) is the terminal point of the path \( F \). For example, for the paths in Figure~\ref{fig402}, we have 
\[
\height(F^{(1)}) = \height(F^{(4)}) = \height(F^{(5)}) = 0, \quad  \height(F^{(2)}) = \height(F^{(3)}) = 1, \quad
\height(F^{(6)}) = 2.
\]
Next, we define the mapping \( f_6: \mathcal{M}_d \to \mathcal{F}_d \) recursively as follows.

\begin{itemize}
    \item If \( M \) is a single indeterminate, then \( f_6(M) \) is the empty path.
    \item If \( M = L(M') \) for some monomial \( M' \in \mathcal{M}_d \), then \( f_6(M) \ce f_6(M'), (0,1) \).
    \item Otherwise, we can write \( M = M_0 M_1 \cdots M_{d-1} \) such that \( M_0 \in \mathcal{M}_d \) and \( M_i \in \overline{\mathcal{M}}_d \) for every \( i \in [d-1] \). Let \( m_i \ce \lofi(M_i) \) for every \( i \in [d-1] \), and let \( m \ce 1 + \sum_{i=1}^{d-1} m_i \). Define
    \[
    f_6(M) \ce F_0, F_1, \ldots, F_{d-1}, \underset{(m_1, \ldots, m_{d-1})}{\underline{(m,1)}}
    \]
    where \( F_0 \ce f_6(M_0) \). For every \( i \in [d-1] \), \( F_i \) is the empty path if \( M_i \) is a single indeterminate; otherwise \( M_i = L(M_i') \) for some \( M_i' \in \mathcal{M}_d \), and we define \( F_i \ce (0,1), f_6(M_i') \).
\end{itemize}

For example, let $M \ce L(a_1L(a_2a_3a_4) L^2(a_5)) \in \M_3$. Then $f_6(M)$ is computed as follows.
\begin{align*}
f_6(M) &=f_6( L(a_1L(a_2a_3a_4) L^2(a_5))) \\
&=   f_6(a_1L(a_2a_3a_4) L^2(a_5)  ) , (0,1)\\
&= (0,1), f_6(a_2a_3a_4), (0,1), f_6(L(a_5)), \underset{(1,2)}{\underline{(4,1)}}, (0,1)\\
&= (0,1), \underset{(0,0)}{\underline{(1,1)}}, (0,1), (0,1), \underset{(1,2)}{\underline{(4,1)}} ,(0,1),
\end{align*}
which is indeed an element in $\F_3$. Notice that $|f_6(M)| = \topt(M) =6$, $\north(f_6(M)) = \lopt(M) =4$, and $\height(f_6(M)) = \lofi(M) = 1$. Also, for the paths given in Figure~\ref{fig402}, we have
\[
\begin{array}{lll}
f_6(a_1a_2a_3a_4a_5) = F^{(1)}, &
f_6(L(a_1a_2a_3)) = F^{(2)}, &
f_6(L(a_1)a_2a_3) = F^{(3)}, \\[3pt]
f_6(a_1L(a_2)a_3) = F^{(4)}, &
f_6(a_1a_2L(a_3)) = F^{(5)}, &
f_6(L(L(a_1))) = F^{(6)}. 
\end{array}
\]
Next, we demonstrate that \( f_6 \) is indeed a bijection.

\begin{theorem}\label{thm402}
Let \( d \geq 2 \) and \( n, k \geq 0 \) be integers. Then \( N_d(n,k) \) counts the number of lattice paths in \( \F_d \) consisting of a total of \( n \) steps, \( k \) of which are \( (0,1) \).
\end{theorem}

\begin{proof}
Let \( M \in \M_d \). Notice that every step in \( f_6(M) \) has the form \( (\ell,1) \) for some \( \ell \geq 0 \), and by construction $f_6(M)$ always stays on or above the line $y=x$. Also, when \( \ell \geq 1 \), \( f_6 \) labels the step \( (\ell,1) \) with a composition of \( \ell-1 \) with \( d-1 \) nonnegative parts. Thus, \( f_6(M) \in \F_d \) for every \( M \in \M_d \).

Next, we prove that \( |f_6(M)| = \topt(M) \) by strong induction on \( \topt(M) \). If \( \topt(M) = 0 \), then \( M \) is a single indeterminate, and \( f_6(M) \) is the empty path, leading to \( |f_6(M)| = 0 = \topt(M) \). Now suppose \( \topt(M) \geq 1 \). If \( M = L(M') \) for some \( M' \in \M_d \), then \( \topt(M') = \topt(M) - 1 \), and by the inductive hypothesis, \( |f_6(M')| = \topt(M') \). Therefore, we have \( f_6(M) = f_6(M'), (0,1) \), which implies
\[
|f_6(M)| = |f_6(M')| + 1 = \topt(M') + 1 = \topt(M).
\]
Otherwise, we can express \( M \) as \( M = M_0 M_1 \cdots M_{d-1} \), where \( M_i \) is irreducible for every \( i \in [d-1] \). Notice that \( \topt(M) = \sum_{i=1}^{d} \topt(M_i) + 1 \). From the definition of \( f_6 \), we have:
\[
f_6(M) = F_0, F_1, \ldots, F_{d-1}, \underset{(m_1, \ldots, m_{d-1})}{\underline{(m,1)}}.
\]
For each \( i \in [d-1] \), either \( F_i \) is the empty path (in which case \( |F_i| = \topt(M_i) = 0 \)), or \( F_i = (0,1), f_6(M_i') \) where \( M_i = L(M_i') \). In the latter case, we know that $|f_6(M_i')| = \topt(M_i')$ by the inductive hypothesis, and so we have
\[
|F_i| = |f_6(M_i')| + 1 = |M_i'| + 1 = |M_i|.
\]
Additionally, since \( F_0 \) is defined as \( f_6(M_0) \), we conclude that \( |F_i| = \topt(M_i) \) for every \( i \in \{0, \ldots, d-1\} \). Hence,
\[
|f_6(M)| = \sum_{i=0}^{d-1} |F_i| + 1 = \sum_{i=0}^{d-1} \topt(M_i) + 1 = \topt(M).
\]
This shows that \( |f_6(M)| = \topt(M) \) in all cases. The same inductive argument can be applied to show that \( \north(f_6(M)) = \lopt(M) \) and \( \height(f_6(M)) = \lofi(M) \) for every \( M \in \M_d \).

Next, we describe the inverse function of \( f_6 \) to complete the proof. Let \( F \in \F_{d} \). If \( F \) is the empty path, then \( f_6^{-1}(F) \) is the monomial with a single indeterminate. Otherwise, we can express \( F \) as \( F = F', (\ell,1) \) for some \( F' \in \F_d\) and \( \ell \geq 0 \). If \( \ell = 0 \), then we have \( f_6^{-1}(F) = L(f_6^{-1}(F')) \). If \( \ell \geq 1 \), the step \( (\ell,1) \) is labelled by a composition \( (\ell_1, \ldots, \ell_{d-1}) \) of \( \ell-1 \). In this case, we note that \( \height(F') = \height(F) + m - 1 \). This implies that for every \( j \in [\height(F) + \ell - 2] \), there exists a \( (0,1) \) step in \( F' \) that transitions from height \( j-1 \) to height \( j \). Then we write
\[
F' = F_0, F_1, \ldots, F_{d-1}
\]
such that for every \( i \in [d-1] \), \( F_i \) is the empty path if \( \ell_i = 0 \); otherwise, the first step of \( F_i \) is the last instance of \( (0,1) \) in \( F' \) that brings the path from height \( \height(F) + \sum_{j=1}^{i-1} \ell_j  \) to \( \height(F) + \sum_{j=1}^{i-1} \ell_j +1 \). By choosing to start at the last instance of \( (0,1) \) at a given height, the path \( F_i \) (as a lattice path in its own right) must start with \( (0,1) \) and cannot subsequently dip below the line \( y=x+1 \). Thus, we can write \( F_i = (0,1), F_i' \) for some \(F_i' \in \F_d\) in this case. Moreover, notice that \( \height(F_i) = \ell_i \) by construction, which also implies that \( \height(F_0) = \height(F) \).

From there, we can express \( f_6^{-1}(F) = M_0 M_1 \cdots M_{d-1} \) as follows. First, \( M_0 = f_6^{-1}(F_0)$. For every \( i \in [d-1] \), if \( F_i \) is the empty path, then \( M_i \) is a single indeterminate. Otherwise, if \( F_i = (0,1), F_i' \) for some \( F_i' \in \F_d \), then \( M_i = L(f_6^{-1}(F_i')) \). Thus, the inverse of \( f_6 \) is well-defined, completing the proof.
\end{proof}

We note that \( \F_3 \) is a slight variant of the \( F \)-paths studied in Huh et al.~\cite{HuhKSS24}. More precisely, the set of eligible steps for our lattice paths in \( \F_3 \) is given by
\[
\set{ \underset{(j, \ell-1-j)}{ \underline{(\ell, 1)}} : \ell \geq 1, j \geq 0} \cup \set{(0,1)}.
\]
In contrast, the set of eligible steps for the original \( F \)-paths from~\cite{HuhKSS24} is
\[
\set{ (a, b) : a \geq 1, b \leq 1} \cup \set{(0,1)}.
\]
In particular, a bijection from \( \F_3 \) to the set of original \( F \)-paths can be established by replacing every step of the form \( \underset{(j, \ell-1-j)}{ \underline{(\ell,1)}} \) with the step \( (\ell - j, 1 - j) \).

Additionally, when \( d=2 \), we observe that every step \( (\ell, 1) \) (where \( \ell \geq 1 \)) is labelled by the one-part composition \( (\ell-1) \). Therefore, we can ignore these markings in this case. Now notice that the lattice path
\[
F \ce (\ell_1, 1), (\ell_2, 1), \ldots, (\ell_n, 1)
\]
belongs to \( \F_2 \) if and only if the following conditions hold:

\begin{itemize}
\item \( \ell_i \geq 0 \) for every \( i \in [n] \), and
\item \( \sum_{j=1}^i \ell_j  \leq i \) for every \( i \in [n] \).
\end{itemize}
(The latter condition ensures that the path remains on or above the line \( y=x \).) Now given a sequence $(\ell_1, \ldots, \ell_n)$ which satisfies the above, consider the sequence $(a_1, \ldots, a_{n+1})$ where $a_i \ce 1 + \sum_{j=1}^{i-1} \ell_j$ for every $i \in [n+1]$. Then we have $1 = a_1 \leq a_2 \leq \cdots \leq a_{n+1}$, with $a_i \leq i$ for every $i \in [n+1]$, which gives a known combinatorial interpretation of the Catalan numbers (see, for instance, \cite[Chapter 2, Exercise 78]{Stanley15}).

\subsection{Labelled Dyck paths}

Given $d \geq 2$, we define $\widetilde{\Q}_d$ to be the set of nonempty Dyck paths such that:
\begin{itemize}
\item
the last descent is unlabelled;
\item
every other descent of length $\ell \geq 1$ is labelled by a composition of $\ell-1$ with $d-1$ nonnegative parts.
\end{itemize}
For instance, an element in $\widetilde{\Q}_3$ is
\[
UUU\underset{(0,1)}{\underline{DD}}UU\underset{(0,0)}{\underline{D}}UU\underset{(2,1)}{\underline{DDDD}}UUDD.
\]
Figure~\ref{fig403} shows all $6$ labelled Dyck paths in $\widetilde{\Q}_3$ with semilength $3$.

\begin{figure}[htbp]
\centering
\def\sc{0.8}
\def\grid{\foreach \x in {\xlb ,...,\xub}
{    \ifthenelse{\NOT 0 = \x}{\draw[thick](\x ,-1pt) -- (\x ,1pt);}{}
\draw[dotted, thick](\x,\ylb- \buf) -- (\x,\yub + \buf);}
\foreach \y in {\ylb ,...,\yub}
{    \ifthenelse{\NOT 0 = \y}{\draw[thick](-1pt, \y) -- (1pt, \y);}{}
\draw[dotted, thick](\xlb- \buf, \y) -- (\xub + \buf, \y);}
}

\begin{tabular}{ccc}
\begin{tikzpicture}[scale = \sc, font=\tiny\sffamily, thick,main node/.style={circle,inner sep=0.4mm,draw, fill}]
\def\xlb{0}; \def\xub{6}; \def\ylb{0}; \def\yub{3}; \def\buf{0}; \grid
\node[main node] at (0,0) (0) {};
\node[main node] at (1,1) (1) {};
\node[main node] at (2,0) (2) {};
\node[main node] at (3,1) (3) {};
\node[main node] at (4,0) (4) {};
\node[main node] at (5,1) (5) {};
\node[main node] at (6,0) (6) {};
\draw[ultra thick] (0) -- (1) --node[pos=0.5, below left, inner sep=0pt] {$(0,0)$} (2) -- (3) --node[pos=0.5, below left, inner sep=0pt] {$(0,0)$} (4) -- (5) -- (6);
\end{tikzpicture}
&
\begin{tikzpicture}[scale = \sc, font=\tiny\sffamily, thick,main node/.style={circle,inner sep=0.4mm,draw, fill}]
\def\xlb{0}; \def\xub{6}; \def\ylb{0}; \def\yub{3}; \def\buf{0}; \grid
\node[main node] at (0,0) (0) {};
\node[main node] at (1,1) (1) {};
\node[main node] at (2,0) (2) {};
\node[main node] at (3,1) (3) {};
\node[main node] at (4,2) (4) {};
\node[main node] at (5,1) (5) {};
\node[main node] at (6,0) (6) {};
\draw[ultra thick] (0) -- (1) --node[pos=0.5, below left, inner sep=0pt] {$(0,0)$} (2) -- (3) -- (4) -- (5) -- (6);
\end{tikzpicture}
&
\begin{tikzpicture}[scale = \sc, font=\tiny\sffamily, thick,main node/.style={circle,inner sep=0.4mm,draw, fill}]
\def\xlb{0}; \def\xub{6}; \def\ylb{0}; \def\yub{3}; \def\buf{0}; \grid
\node[main node] at (0,0) (0) {};
\node[main node] at (1,1) (1) {};
\node[main node] at (2,2) (2) {};
\node[main node] at (3,1) (3) {};
\node[main node] at (4,2) (4) {};
\node[main node] at (5,1) (5) {};
\node[main node] at (6,0) (6) {};
\draw[ultra thick] (0) -- (1) -- (2) --node[pos=0.5, below left, inner sep=0pt] {$(0,0)$} (3) -- (4) -- (5) -- (6);
\end{tikzpicture}
\\
$Q^{(1)}$&
$Q^{(2)}$&
$Q^{(3)}$
\\
\\
\begin{tikzpicture}[scale = \sc, font=\tiny\sffamily, thick,main node/.style={circle,inner sep=0.4mm,draw, fill}]
\def\xlb{0}; \def\xub{2}; \def\ylb{0}; \def\yub{2}; \def\buf{0}; \grid
\def\xlb{0}; \def\xub{6}; \def\ylb{0}; \def\yub{3}; \def\buf{0}; \grid
\node[main node] at (0,0) (0) {};
\node[main node] at (1,1) (1) {};
\node[main node] at (2,2) (2) {};
\node[main node] at (3,1) (3) {};
\node[main node] at (4,0) (4) {};
\node[main node] at (5,1) (5) {};
\node[main node] at (6,0) (6) {};
\draw[ultra thick] (0) -- (1) -- (2) -- (3) -- (4) -- (5) -- (6);
\draw (2) --node[pos=0.5, below left, inner sep=0pt] {$(1,0)$} (4);
\end{tikzpicture}
&
\begin{tikzpicture}[scale = \sc, font=\tiny\sffamily, thick,main node/.style={circle,inner sep=0.4mm,draw, fill}]
\def\xlb{0}; \def\xub{6}; \def\ylb{0}; \def\yub{3}; \def\buf{0}; \grid
\node[main node] at (0,0) (0) {};
\node[main node] at (1,1) (1) {};
\node[main node] at (2,2) (2) {};
\node[main node] at (3,1) (3) {};
\node[main node] at (4,0) (4) {};
\node[main node] at (5,1) (5) {};
\node[main node] at (6,0) (6) {};
\draw[ultra thick] (0) -- (1) -- (2) -- (3) -- (4) -- (5) -- (6);
\draw (2) --node[pos=0.5, below left, inner sep=0pt] {$(0,1)$} (4);
\end{tikzpicture}

&
\begin{tikzpicture}[scale = \sc, font=\tiny\sffamily, thick,main node/.style={circle,inner sep=0.4mm,draw, fill}]
\def\xlb{0}; \def\xub{6}; \def\ylb{0}; \def\yub{3}; \def\buf{0}; \grid
\node[main node] at (0,0) (0) {};
\node[main node] at (1,1) (1) {};
\node[main node] at (2,2) (2) {};
\node[main node] at (3,3) (3) {};
\node[main node] at (4,2) (4) {};
\node[main node] at (5,1) (5) {};
\node[main node] at (6,0) (6) {};
\draw[ultra thick] (0) -- (1) -- (2) -- (3) -- (4) -- (5) -- (6);
\end{tikzpicture}
\\
$Q^{(4)}$ &
$Q^{(5)}$ &
$Q^{(6)}$ 
\end{tabular}
\caption{The $6$ labelled Dyck paths $Q \in \widetilde{Q}_3$ with $|Q| = 3$}\label{fig403}
\end{figure}

We show that there is a very simple bijection between $\F_d$ and $\widetilde{\Q}_d$. Given $F \in \F_d$ where
\[
F = (\ell_1, 1), (\ell_2, 1), \ldots, (\ell_n, 1), 
\]
define
\[
f_7(F) \ce UD^{\ell_1} UD^{\ell_2}\cdots UD^{\ell_n} UD^{n+1-\sum_{i=1}^n \ell_i}.
\]
Moreover, for every $i \in [n]$ where $\ell_i \geq 1$, we mark the descent $D^{\ell_i}$ in $f_7(F)$ by the same composition that labelled the step $(\ell_i, 1)$ in $F$. For example, consider $F \in \F_3$ where
\[
F \ce (0,1), \underset{(0,0)}{\underline{(1,1)}}, (0,1), (0,1), \underset{(1,2)}{\underline{(4,1)}} ,(0,1).
\]
Then
\[
f_7(F) = UU\underset{(0,0)}{\underline{D}}UUU\underset{(1,2)}{\underline{DDDD}}UUDD.
\]
Also, for the lattice paths in Figure~\ref{fig402} and labelled Dyck paths in Figure~\ref{fig403}, we have $f_7(F^{(i)}) = Q^{(i)}$ for every $i \in [6]$. 
Then we have the following.

\begin{theorem}\label{thm403}
Let $d \geq 2$ and $n,k \geq 0$ be integers. Then $N_d(n,k)$ counts the number of labelled Dyck paths in $\widetilde{\Q}_d$ with semilength $n+1$ and $k$ instances of $UU$.
\end{theorem}

\begin{proof}
Given $F \in \F_d$, where
\[
F = (\ell_1, 1), (\ell_2, 1), \ldots, (\ell_n, 1), 
\]
we know by the definition of $\F_d$ that $\sum_{i=1}^j \ell_i \leq j$ for every $j \in [n]$, which assures that $f_7(F)$ never contains a prefix with more down steps than up steps. Also, given that $F$ contains $n$ steps,  $f_7(F)$ must contain exactly $n+1$ up steps and $n+1$ down steps, and thus must indeed be a Dyck path with semilength $n+1$. Thus, $|f_7(F)| = |F| + 1$. Also, observe that the $i$-th instance of $U$ in $f_7(F)$ is followed by another $U$ if and only if $\ell_i = 0$. Thus, we see that the number of instances of $UU$ in $f_7(F)$ is exactly equal to the number of instances of $(0,1)$ in $F$. Furthermore, since the composition labellings are unaltered by $f_7$, we see that $f_7(F) \in \widetilde{\Q}_d$ for every $F \in \F_d$. Finally, it is rather straightforward to see that $f_7$ is reversible, and thus the claim follows.
\end{proof}

We remark that, instead of labelling each descent of length $\ell$ with a composition $(\ell_1, \ldots, \ell_{d-1})$ which sums to $\ell-1$, we can equivalently think of allowing $d-1$ possible types of down steps $D_1, \ldots, D_{d-1}$, and requiring that each descent (except the last one in the path) of length $\ell$ has the form
\[
D_1^{\ell_1}D_2^{\ell_2} \cdots D_{d-1}^{\ell_{d-1}+1}.
\]
This gives a $(d-1)$-coloured Dyck path interpretation of $\widetilde{\Q}_d$. When $d=3$, this specializes to the restricted bi-coloured Dyck paths studied in B\'{e}nyi et al.~\cite{BenyiMR24}, Huh et al.~\cite{HuhKSS24}, and Yan and Lin~\cite{YanL20}.

Moreover, when $d=2$, there is only $d-1=1$ type of down steps, and $\widetilde{\Q}_2$ is simply the set of (unlabelled and nonempty) Dyck paths. In this case, Theorem~\ref{thm403} assures that $N_2(n,k)$ counts the number of Dyck paths with semilength $n+1$ and $k$ instances of $UU$. While we cannot find a reference mentioning this exact fact, we offer a simple independent proof of it using known properties about Narayana numbers and Dyck paths.

\begin{corollary}\label{cor403}
Let  $n,k \geq 0$ be integers. Then $N_2(n,k)$ counts the number of Dyck paths with semilength $n+1$ and $k$ instances of $UU$.
\end{corollary}

\begin{proof}
Recall that $N_2(n,k)$ counts the number of Dyck paths with semilength $n+1$ and $k+1$ peaks~\cite[Section 2.4.2]{Petersen15}. Also, in any Dyck path $Q$, every up step is either followed by another up step (which creates an instance of $UU$) or a down step (which forms a peak). Thus, the number of $UU$ and peak instances must sum to $|Q|$. Hence, we see that $N_2(n,k)$ counts the number of Dyck paths with semilength $n+1$ and $n-k$ instances of $UU$. Finally, it is easy to see from the formula of $N_2(n,k)$ that $N_2(n,k) = N_2(n,n-k)$ for every $n,k \geq 0$. Thus, the claim follows.
\end{proof}

Since $N_d(n,k) \neq N_d(n,n-k)$ in general, the proof of Corollary~\ref{cor403} seemingly cannot be extended to translate Theorem~\ref{thm403} into counting paths in $\widetilde{\Q}_d$ in terms of the number of peaks when $d \geq 3$.

\subsection{Labelled ordered trees}

As is the case for the mapping $f_3 : \T_d \to \Q_d$ discussed in Section~\ref{sec303}, herein we leverage similar ideas to obtain a labelled ordered tree interpretation of $N_d(n,k)$ by applying a path-to-tree mapping from $\widetilde{\Q}_d$. Given $d \geq 2$, we define  $\widetilde{\T}_d$ to be the set of labelled ordered trees such that:
\begin{itemize}
\item
the root node is an internal node, and is unlabelled;
\item
every non-root internal node with outdegree $\ell \geq 1$ is labelled by a composition of $\ell-1$ with $d-1$ nonnegative parts.
\end{itemize}
Figure~\ref{fig404} illustrates the $6$ trees in $\widetilde{\T}_d$ with $3$ edges.

\begin{figure}[htbp]
\centering
\def\sc{1}
\def\ysc{0.8}
\begin{tabular}{cccccc}

\begin{tikzpicture}[xscale=\sc, yscale = \ysc,thick,main node/.style={circle,inner sep=0.5mm,draw,font=\small\sffamily}, word node/.style={font=\scriptsize}
]
\node[main node] at (2,2) (0) {};
\node[main node] at (2,1) (1) {};
\node[main node] at (2,0) (2) {};
\node[main node] at (2,-1) (3) {};

\node[word node, align=left, right] at (2,1) {$(0,0)$};
\node[word node, align=left, right] at (2,0) {$(0,0)$};

 \path[every node/.style={font=\sffamily}]
(0) edge (1) (1) edge (2) (2) edge (3);
\end{tikzpicture}

&
\begin{tikzpicture}[xscale=\sc, yscale = \ysc,thick,main node/.style={circle,inner sep=0.5mm,draw,font=\small\sffamily}, word node/.style={font=\scriptsize}
]
\node[main node] at (2,2) (0) {};
\node[main node] at (1.5,1) (1) {};
\node[main node] at (2.5,1) (2) {};
\node[main node] at (2.5,0) (3) {};

\node[word node, align=left, right] at (2.5,1) {$(0,0)$};

 \path[every node/.style={font=\sffamily}]
(0) edge (1) (0) edge (2) (2) edge (3);
\end{tikzpicture}

&

\begin{tikzpicture}[xscale=\sc, yscale = \ysc,thick,main node/.style={circle,inner sep=0.5mm,draw,font=\small\sffamily}, word node/.style={font=\scriptsize}
]
\node[main node] at (2,2) (0) {};
\node[main node] at (1.5,1) (1) {};
\node[main node] at (2.5,1) (2) {};
\node[main node] at (1.5,0) (3) {};

\node[word node, align=right, left] at (1.5,1) {$(0,0)$};

 \path[every node/.style={font=\sffamily}]
(0) edge (1) (0) edge (2) (1) edge (3);
\end{tikzpicture}
&

\begin{tikzpicture}[xscale=\sc, yscale = \ysc,thick,main node/.style={circle,inner sep=0.5mm,draw,font=\small\sffamily}, word node/.style={font=\scriptsize}
]
\node[main node] at (2,2) (0) {};
\node[main node] at (2,1) (1) {};
\node[main node] at (1.5,0) (2) {};
\node[main node] at (2.5,0) (3) {};

\node[word node, align=left, right] at (2,1) {$(1,0)$};

 \path[every node/.style={font=\sffamily}]
(0) edge (1) (1) edge (2) (1) edge (3);
\end{tikzpicture}

&

\begin{tikzpicture}[xscale=\sc, yscale = \ysc,thick,main node/.style={circle,inner sep=0.5mm,draw,font=\small\sffamily}, word node/.style={font=\scriptsize}
]
\node[main node] at (2,2) (0) {};
\node[main node] at (2,1) (1) {};
\node[main node] at (1.5,0) (2) {};
\node[main node] at (2.5,0) (3) {};

\node[word node, align=left, right] at (2,1) {$(0,1)$};

 \path[every node/.style={font=\sffamily}]
(0) edge (1) (1) edge (2) (1) edge (3);
\end{tikzpicture}

&

\begin{tikzpicture}[xscale=\sc, yscale = \ysc,thick,main node/.style={circle,inner sep=0.5mm,draw,font=\small\sffamily}, word node/.style={font=\scriptsize}
]
\node[main node] at (2,2) (0) {};
\node[main node] at (1,1) (1) {};
\node[main node] at (2,1) (2) {};
\node[main node] at (3,1) (3) {};


 \path[every node/.style={font=\sffamily}]
(0) edge (1) (0) edge (2) (0) edge (3);
\end{tikzpicture}

\\

$T^{(1)}$ &
$T^{(2)}$ &
$T^{(3)}$ &
$T^{(4)}$ &
$T^{(5)}$ &
$T^{(6)}$ 
\end{tabular}
\caption{The $6$ trees in $\widetilde{\T}_3$ with $3$ edges}
\label{fig404}
\end{figure}

We define a bijection $f_8 : \widetilde{\Q}_d \to \widetilde{\T}_d$ as follows. Given $Q \in \widetilde{\Q}_d$ with semilength $n$, find integers $\ell_1, \ldots, \ell_n$ where
\[
Q = UD^{\ell_1} UD^{\ell_2} \cdots UD^{\ell_n}.
\]
Then we let $f_8(Q)$ be the unique ordered tree on $n+1$ vertices for which the sequence of vertex outdegrees read in preorder traversal is $(\ell_n, \ldots, \ell_1, 0)$. Moreover, for every $i \in [n-1]$, the non-root internal node in $f_8(Q)$ with outdegree $\ell_i$ is labelled by the same composition that labelled the descent $D^{\ell_i}$ in $Q$. For example, consider the labelled Dyck paths in Figure~\ref{fig403} and labelled ordered trees in Figure~\ref{fig404}. Then $f_8(Q^{(i)}) = T^{(i)}$ for every $i \in [6]$. Then we have the following.

\begin{theorem}\label{thm404}
Let $d \geq 2$ and $n,k \geq 0$ be integers. Then $N_d(n,k)$ counts the number of labelled ordered trees in $\widetilde{\T}_d$ with $n+1$ edges and $k+1$ leaves.
\end{theorem}

\begin{proof}
Suppose we are given $Q \ce UD^{\ell_1} UD^{\ell_2} \cdots UD^{\ell_n} \in \widetilde{\Q}_d$ where, for every $i \in [n-1]$, the descent $D^{\ell_i}$ is labelled by a composition of $\ell-1$ with $d-1$ nonnegative parts. By construction, $f_8(Q)$ is a tree whose root node has outdegree $\ell_n$. Since the last descent in $Q$ is unlabelled, so is the root node of $f_8(Q)$. Also, since $f_8(Q)$ has $|Q|+1$ vertices, it must have $|Q|$ edges. Next, the number of leaves in $f_8(Q)$ is equal to the number of zero entries in the sequence $(\ell_n, \ldots, \ell_1, 0)$, which in turn is equal to the number of instances of $(0,1)$ in $Q$ plus 1.

The inverse mapping of $f_8$ is also straightforward. Given $T \in \widetilde{\T}_d$, one lists the outdegrees of its vertices in preorder traversal, omitting the last entry (which must be zero), and then apply the path-to-tree mapping $f_3$ defined in Section~\ref{sec303}. This results in a Dyck path $Q'$ where the first ascent is unmarked, while every other ascent of length $\ell$ is labelled by a composition of $\ell-1$ with $d-1$ parts. Then we reflect $Q'$ horizontally by swapping its up and down steps and reading the steps in reverse order, which results in the labelled Dyck path $f_8^{-1}(T) \in \widetilde{\Q}_d$.
\end{proof}

Notice that when $d=3$, each non-root internal node with outdegree $\ell \geq 1$ has $\ell$ possible labels: $(0, \ell-1), (1,\ell-2), \ldots, (\ell-1,0)$. Thus, $\F_3$ can be alternatively seen as the set of ordered trees in which every non-root internal node is labelled by a positive integer less than or equal to its outdegree. This set of trees has been studied in B\'{e}nyi et al.~\cite{BenyiMR24}, Huh et al.~\cite{HuhKSS24}, and Yan and Lin~\cite{YanL20}. When $d=2$, the marking of each internal node is determined by its outdegree, and thus we might consider the trees in $\widetilde{\T}_2$ unlabelled. Thus, as with Theorem~\ref{thm303}, Theorem~\ref{thm404} implies that $N_2(n,k)$ counts the number of unlabelled ordered trees with $n+1$ edges and $k+1$ leaves.

\section{Some future research directions}\label{sec5}

In this manuscript, we introduced the generalized Narayana numbers $N_d(n,k)$, 
and saw nine combinatorial interpretations of them, 
some of which generalize known interpretations of $N_3(n,k)$. 
On the other hand, there are a number of other combinatorial objects which correspond to 
$C_3(n)$ and $N_3(n,k)$ that we have not discussed. 
For instance, it is known that $C_3(n)$ counts the following sets:
\begin{itemize}
 \item
permutations of $[n+1]$ which avoid the patterns $4123$, $4132$, and $4213$~\cite{AlbertHPSV18};
\item
permutations of $[n+1]$ which avoid the patterns $2341$, $2431$, and $3241$~\cite{CallanM18, GaoK19, HuhKSS24};
 \item
 inversion sequences of length $n+1$ (i.e.,  $(\ell_1, \ldots, \ell_{n+1})$ where $0 \leq \ell_i \leq i-1$ for every $i \in [n+1]$) which avoid the patterns $101$ and $102$~\cite{CaoJL19, HuhKSS24};
 \item
 inversion sequences of length $n+1$ which avoid the patterns $101$ and $021$~\cite{HuhKSS24, YanL20}.
 \end{itemize}
 
Can any of these results be generalized to objects counted by $C_d(n)$ and $N_d(n,k)$? Since $C_2(n)$ and $C_3(n)$ are known to count certain subsets of permutations on $[n+1]$, it readily follows that $C_d(n) \leq (n+1)!$ holds for all $n \geq 0$ when $d \in \set{2,3}$. However,  the inequality fails starting at $d =4$ (e.g., $C_4(2) = 7 > 3!$). Thus, if a generalization exists, it would likely involve permutations or inversion sequences of size greater than $n+1$, and/or the use of differentiating labellings.
  
\bibliographystyle{plain}
\bibliography{ref} 

\begin{thebibliography}{10}

\bibitem{AlbertHPSV18}
Michael~H. Albert, Cheyne Homberger, Jay Pantone, Nathaniel Shar, and Vincent
  Vatter.
\newblock Generating permutations with restricted containers.
\newblock {\em J. Combin. Theory Ser. A}, 157:205--232, 2018.

\bibitem{AlloucheJ96}
Jean-Paul Allouche and Tom Johnson.
\newblock {Narayana's Cows and Delayed Morphisms}.
\newblock In {\em {Journ{\'e}es d'Informatique Musicale}}, {\^i}le de Tatihou,
  France, May 1996.

\bibitem{AsinowskiB24}
Andrei Asinowski and Cyril Banderier.
\newblock From geometry to generating functions: rectangulations and
  permutations.
\newblock {\em S\'em. Lothar. Combin.}, 91B:Art. 46, 12, 2024.

\bibitem{Barry11}
Paul Barry.
\newblock On a generalization of the {N}arayana triangle.
\newblock {\em J. Integer Seq.}, 14(4):Article 11.4.5, 22, 2011.

\bibitem{Barry16}
Paul Barry.
\newblock Riordan arrays, generalized {N}arayana triangles, and series
  reversion.
\newblock {\em Linear Algebra Appl.}, 491:343--385, 2016.

\bibitem{BeatonBGR19}
Nicholas~R. Beaton, Mathilde Bouvel, Veronica Guerrini, and Simone Rinaldi.
\newblock Slicing of parallelogram polyominoes: {C}atalan, {S}chr\"oder,
  {B}axter, and other sequences.
\newblock {\em Electron. J. Combin.}, 26(3):Paper No. 3.13, 36, 2019.

\bibitem{BenyiMR24}
Be\'ata B\'enyi, Toufik Mansour, and Jos\'e{}~L. Ram\'irez.
\newblock Pattern avoidance in weak ascent sequences.
\newblock {\em Discrete Math. Theor. Comput. Sci.}, 26(1):Paper No. 2, 16,
  [2024--2025].

\bibitem{Bremner25}
Murray~R. Bremner.
\newblock Algebraic identities for linear operators on associative triple
  systems (long version).
\newblock {\em \textrm{\url{arXiv:2512.04910}}}, 2025.

\bibitem{BremnerDotsenko}
Murray~R. Bremner and Vladimir Dotsenko.
\newblock {\em Algebraic Operads: An Algorithmic Companion}.
\newblock CRC Press, Taylor \& Francis Group, 2016.

\bibitem{BremnerE22}
Murray~R. Bremner and Hader~A. Elgendy.
\newblock A new classification of algebraic identities for linear operators on
  associative algebras.
\newblock {\em J. Algebra}, 596:177--199, 2022.

\bibitem{BrycFM19}
W{\l}odzimierz Bryc, Raouf Fakhfakh, and Wojciech M{\l}otkowski.
\newblock Cauchy-{S}tieltjes families with polynomial variance functions and
  generalized orthogonality.
\newblock {\em Probab. Math. Statist.}, 39(2):237--258, 2019.

\bibitem{BursteinS22}
Alexander Burstein and Louis~W. Shapiro.
\newblock Pseudo-involutions in the {R}iordan group.
\newblock {\em J. Integer Seq.}, 25(3):Art. 22.3.6, 54, 2022.

\bibitem{Callan22}
David Callan.
\newblock A note on generalized {N}arayana numbers.
\newblock {\em \url{arXiv:2205.08277}}, 2022.

\bibitem{CallanM18}
David Callan and Toufik Mansour.
\newblock Enumeration of small {W}ilf classes avoiding 1342 and two other
  4-letter patterns.
\newblock {\em Pure Math. Appl. (PU.M.A.)}, 27(1):62--97, 2018.

\bibitem{CaoJL19}
Wenqin Cao, Emma~Yu Jin, and Zhicong Lin.
\newblock Enumeration of inversion sequences avoiding triples of relations.
\newblock {\em Discrete Appl. Math.}, 260:86--97, 2019.

\bibitem{Carlsson}
Renate Carlsson.
\newblock $n$-ary algebras.
\newblock {\em Nagoya Mathematical Journal}, 78:45--56, 1980.

\bibitem{GaoK19}
Alice L.~L. Gao and Sergey Kitaev.
\newblock On partially ordered patterns of lengths 4 and 5 in permutations.
\newblock {\em Electron. J. Combin.}, 26(3):Paper No. 3.26, 31, 2019.

\bibitem{Gnedbaye}
Allahtan~Victor Gnedbaye.
\newblock Op\'erades des algebres $(k+1)$-aires.
\newblock {\em Operads: Proceedings of Renaissance Conferences, Contemporary
  Mathematics (American Mathematical Society)}, 202:83--113, 1997.

\bibitem{Grimaldi12}
Ralph~P. Grimaldi.
\newblock {\em Fibonacci and {C}atalan numbers: An introduction}.
\newblock John Wiley \& Sons, Inc., Hoboken, NJ, 2012.

\bibitem{GuLM08}
Nancy S.~S. Gu, Nelson~Y. Li, and Toufik Mansour.
\newblock 2-binary trees: bijections and related issues.
\newblock {\em Discrete Math.}, 308(7):1209--1221, 2008.

\bibitem{HuhKSS24}
JiSun Huh, Sangwook Kim, Seunghyun Seo, and Heesung Shin.
\newblock Bijections on pattern avoiding inversion sequences and related
  objects.
\newblock {\em Adv. in Appl. Math.}, 161:Paper No. 102771, 40, 2024.

\bibitem{KruchininKS22}
Dmitry Kruchinin, Vladimir Kruchinin, and Yuriy Shablya.
\newblock On some properties of generalized {N}arayana numbers.
\newblock {\em Quaest. Math.}, 45(12):1949--1963, 2022.

\bibitem{LodayVallette}
Jean-Louis Loday and Bruno Vallette.
\newblock {\em Algebraic Operads}.
\newblock Grundlehren der mathematischen Wissenschaften, 346. Springer Berlin,
  Heidelberg, 2012.

\bibitem{MartinezS18}
Megan Martinez and Carla Savage.
\newblock Patterns in inversion sequences {II}: inversion sequences avoiding
  triples of relations.
\newblock {\em J. Integer Seq.}, 21(2):Art. 18.2.2, 44, 2018.

\bibitem{MerinoM23}
Arturo Merino and Torsten M\"utze.
\newblock Combinatorial generation via permutation languages. {III}.
  {R}ectangulations.
\newblock {\em Discrete Comput. Geom.}, 70(1):51--122, 2023.

\bibitem{OEIS}
{OEIS Foundation Inc.}
\newblock {The On-Line Encyclopedia of Integer Sequences}, 2025.
\newblock Published electronically at \url{http://oeis.org}.

\bibitem{Petersen15}
T.~Kyle Petersen.
\newblock {\em Eulerian numbers}.
\newblock Birkh\"auser Advanced Texts: Basler Lehrb\"ucher. [Birkh\"auser
  Advanced Texts: Basel Textbooks]. Birkh\"auser/Springer, New York, 2015.
\newblock With a foreword by Richard Stanley.

\bibitem{Stanley15}
Richard~P. Stanley.
\newblock {\em Catalan numbers}.
\newblock Cambridge University Press, New York, 2015.

\bibitem{Sulanke04}
Robert~A. Sulanke.
\newblock Generalizing {N}arayana and {S}chr\"oder numbers to higher
  dimensions.
\newblock {\em Electron. J. Combin.}, 11(1):Research Paper 54, 20, 2004.

\bibitem{Wilf06}
Herbert~S. Wilf.
\newblock {\em generatingfunctionology}.
\newblock A K Peters, Ltd., Wellesley, MA, third edition, 2006.

\bibitem{YanL20}
Chunyan Yan and Zhicong Lin.
\newblock Inversion sequences avoiding pairs of patterns.
\newblock {\em Discrete Math. Theor. Comput. Sci.}, 22(1):Paper No. 23, 35,
  [2020--2021].

\end{thebibliography}

\end{document}